\documentclass[a4paper]{amsart}
\usepackage{amsmath}
\usepackage{hyperref}

\newcommand{\N}{\mathbb{N}}
\newcommand{\R}{\mathbb{R}}

 \newcommand{\C}{\mathbb{C}}
\newcommand{\I}{\mathbb{I}}
\newcommand{\settc}[2]{\bigl\{\,#1 \bigm\vert #2\,\bigr\}}

\newcommand{\norm}[1]{\lVert #1 \rVert}
\newcommand{\dalfa}{\boldsymbol\alpha}
\newcommand{\cfstr}{\alpha_{\text{fs}}}

\newcommand{\mult}{\mu}

\newcommand{\free}{D_{0}}

\newcommand{\fw}{U_{_{\text{FW}}}}

\DeclareMathOperator{\RE}{Re}

\newcommand{\abs}[1]{\lvert #1 \rvert}

\DeclareMathOperator{\spann}{span}

\newtheorem{thm}[equation]{Theorem}
\newtheorem{cor}[equation]{Corollary}
\newtheorem{prop}[equation]{Proposition}
\newtheorem{lem}[equation]{Lemma}

\theoremstyle{remark}
\newtheorem{rem}[equation]{Remark}

\numberwithin{equation}{section}

\title[Relativistic one electron atom]{Ground state for the relativistic 
one electron atom}

\author[Coti Zelati]{Vittorio Coti Zelati}

\email[Coti Zelati]{zelati@unina.it}

\address[Coti Zelati]{Dipartimento di Matematica Pura e Applicata 
``R.~Caccioppoli''\\
Universit\`a di Napoli ``Federico II''\\
via Cintia, M.S.~Angelo\\
80126 Napoli (NA), Italy}

\author{Margherita Nolasco} 

\email[Nolasco]{nolasco@univaq.it}

\address[Nolasco]{Dipartimento di Ingegneria e Scienze
dell'informazione e Matematica\\
Universit\`a dell'Aquila\\
via Vetoio, Loc.~Coppito\\
67010 L'Aquila (AQ) Italia}

\thanks{Research partially supported by MIUR grant PRIN 2015
2015KB9WPT, ``Variational methods, with applications to problems in
mathematical physics and geometry''.}

\begin{document}

\begin{abstract} 
	We study the Dirac-Maxwell system coupled with an external
	potential of Coulomb type.  We use the Foldy--Wouthuysen (unitary)
	transformation of the Dirac operator and its realization as an
	elliptic problem in the 4-dim half space $\mathbb{R}^4_{+}$ with
	Neumann boundary condition.  Using this approach we study the
	existence of a ``ground state'' solution.
\end{abstract} 

\maketitle

\section{Introduction and main results}

The Dirac operator is a first order operator acting on the 4-spinors
$\psi \colon \mathbb{R}^{3} \to \mathbb{C}^{4}$ describing a
relativistic electron given by
\begin{equation*}
	D_{0}= - i c \hbar \dalfa \cdot \nabla + m c^2 \mathbf{\beta}
\end{equation*}
Here $c$ denotes the speed of light, $m > 0$ the mass, $\hbar$ the
Planck's constant, $\dalfa =(\alpha_1,\alpha_2,\alpha_3)$ and $\beta$
are the Pauli-Dirac $4 \times 4$-matrices,
\begin{equation*}
	\mathbf{\beta } = 
	\begin{pmatrix} 
		{\mathbb{I}}_{_2} & 0_{_2} \\
		0_{_2} & - \mathbb{I}_{_2}
	\end{pmatrix} 
	\qquad \mathbf{\alpha}_k =
	\begin{pmatrix} 
		0_{_2} & \mathbf{ \sigma}_k \\
		\mathbf{ \sigma}_k & 0_{_2} 
	\end{pmatrix} \qquad k=1,2,3
\end{equation*}
and $\mathbf{\sigma}_k$ are the Pauli $2\times 2$-matrices.  We take
units such that $m = c = \hbar = 1$.  We are interested in perturbed
Dirac operators $D_{0} + \cfstr V$, $V$ being a Coulomb potential,
$V(x) = - \frac{Z}{|x|}$, $\cfstr = \frac{e^{2}}{\hbar c} \approx
\frac{1}{137}$ is the dimensionless fine structure constant and $Z$,
positive integer, is the atomic number.

Due to the unboundedness of the spectrum of the free Dirac operator,
many efforts have been devoted to the characterization and computation
of the eigenvalues for the Dirac-Coulomb Hamiltonian $D_0 + \cfstr V$,
see \cite{Dolbeault_Esteban_Sere_2000} and references therein.

Here we add the interaction of the electron with its own (static)
electromagnetic field.  The scalar potential $\Phi$ and the vector
potential $A = (A_1,A_2,A_3)$ of the electromagnetic field generated
by the electron $\psi$ satisfy the following (static) Maxwell
equations
\begin{equation*}
	- \Delta \Phi = 4\pi \rho ; \qquad \qquad - \Delta A= {4\pi} J
\end{equation*}
where $\rho = |\psi|^2$ is the charge density and $J = (\psi, \dalfa
\psi )$ the current of the electron.  Therefore
\begin{equation*}
	\Phi = |\psi|^2 \ast \frac{1}{|x|} \qquad \text{and} \qquad A =
	(\psi, \dalfa \psi ) \ast \frac{1}{|x|}.
\end{equation*}

The interaction is obtained through the minimal coupling prescription,
which has, in our units, the following form
\begin{equation*}
	D \psi  = \dalfa \cdot (-i\nabla - \cfstr A) \psi + \cfstr \Phi \psi + 
	\beta \psi + \cfstr V \psi 
\end{equation*}

We have the following result
\begin{thm}
	\label{thm:main} 
	Let $V(x) =- \frac{Z}{|x|}$ with $Z \in \mathbb{N}$ the atomic
	number.  For any $4 < Z < 124 $ there exists $\mult \in (0, 1)$
	and $\psi \in H^{1/2}(\mathbb{R}^3; \mathbb{C}^4) \cap_{1 \leq q <
	3/2} W^{1,q}_{\text{loc}}(\mathbb{R}^3; \mathbb{C}^4)$ a solution
	of the following Maxwell-Dirac eigenvalue problem
	\begin{equation}
		\label{eq:MCD}
		\tag{MDC}
		\begin{cases} 
			\dalfa \cdot (-i\nabla - \cfstr A) \psi + \cfstr \Phi \psi
			+ \beta \psi + \cfstr V \psi = \mult \psi \\
			\abs{\psi}^{2}_{L^{2}} = 1 \\
			- \Delta \Phi = 4\pi\rho = 4\pi\abs{\psi}^{2} \qquad -
			\Delta A = 4\pi J = 4\pi (\psi, \dalfa \psi)
		\end{cases} 
	\end{equation}
	Moreover $ (\psi, \mult )$ is (up to phase) the state of
	lowest positive energy of the system (``ground state'').
\end{thm}

This existence result is strictly related to the results in
\cite{EstebanSere99}, where the Authors consider the Dirac-Fock
equations for Atoms and Molecules.  The equation considered in that
article describe an atom (even a molecule) with a (fixed) nucleus and
N electrons, and takes into account the interaction of each electron
with the nucleus and the other electrons, but not the interaction of
the electrons with their own electric and magnetic field.  Using the
Hartree approximation one ends with an equation similar to the one for
the atom with one electron that we consider in our model
\eqref{eq:MCD}.

Let us also point out that we will prove our result via variational
methods, after performing a unitary change of variables (the
Foldy-Wouthuysen transformation) and a reduction of the problem to an
elliptic problem in the 4-dim half space $\R^4_{+}$ with nonlinear
Neumann boundary condition.

Even in this different setting, we have used in the analysis of the
variational structure of the problem some ideas contained in
\cite{EstebanSere99, Dolbeault_Esteban_Sere_2000, Morozov_Muller2015}.

\section{The FW transformation and the Dirichlet to Neumann operator}

Let us recall first the main properties of the free Dirac operator
$\free = -i\dalfa \cdot \nabla + \beta $ (see
e.g.~\cite{Thaller_1992}).  $\free$ is essentially self-adjoint on
$C_{0}^{\infty}(\mathbb{R}^3 \setminus \{ 0 \};\mathbb{C}^{4})$ and
self-adjoint on $\mathcal{D}(\free) = H^{1}(\mathbb{R}^{3};
\mathbb{C}^{4})$.  Its spectrum is purely absolutely continuous and it
is given by
\begin{equation*}
	\sigma(\free) = (-\infty,-1] \cup [1, +\infty).
\end{equation*}
Let define $\mathcal{Q}_{D_0} \colon
H^{1/2}(\mathbb{R}^{3};\mathbb{C}^{4}) \times
H^{1/2}(\mathbb{R}^{3};\mathbb{C}^{4}) \to \C$ the sesquilinear form
associated to the operator $D_0$.

Let denote by $ \hat{u}$ or $ \mathcal{F}(u)$ the Fourier transform
extending the formula
\begin{equation*} 
	\hat{u}(p) = \frac{1}{(2\pi)^{3/2}} \int_{\mathbb{R}^{3}} e^{-ip
	\cdot x} u(x) \, dx, \qquad \text{for } \, u \in
	\mathcal{S}(\mathbb{R}^{3}).
\end{equation*}
In the (momentum) Fourier space the free Dirac operator is given by
the multiplication operator $\hat{D}_{0} (p) = \mathcal{F} D_{0}
\mathcal{F}^{-1} = \dalfa \cdot p + \beta$ that is for each $p \in
\R^3$ an Hermitian $4 \times 4$-matrix with eigenvalues
\begin{equation*}
	\lambda_1(p) = \lambda_2(p) = - \lambda_3(p)= - \lambda_4(p) =
	\sqrt{|p|^2 + 1} \equiv \lambda(p).
\end{equation*}

The unitary transformation $U(p)$ which diagonalize $\hat{D}_{0}(p)$
is given explicitly by
\begin{align*}
	& U(p) = a_{+} (p) \I_{4} + a_{-} (p) \beta
	\frac{\dalfa \cdot p}{|p|} \\
    &U^{-1}(p) = a_{+} (p) \I_{4} - a_{-} (p) \beta
    \frac{\dalfa \cdot p}{|p|}
\end{align*}
with $a_{\pm}(p) = \sqrt{\frac{1}{2}(1 \pm \frac{1}{\lambda(p)})}$, we
have
\begin{equation*}
	U(p)\hat{D}_{0}(p) U^{-1}(p) = \lambda(p) \beta = \sqrt{|p|^2 + 1}
	\, \beta.
\end{equation*}

Hence there are two orthogonal projectors $ \Lambda_{\pm} $ on
$L^{2}(\mathbb{R}^{3}, \mathbb{C}^{4})$ , both with infinite rank,
given by
\begin{equation}
	\label{eq:projectors}
	\Lambda_{\pm} = \mathcal{F}^{-1}U(p)^{-1} \left( \frac{\I_{4} \pm
	\mathbf{\beta}}{2} \right) U(p)\mathcal{F}.
\end{equation}
such that
\begin{equation*}
	\free \Lambda_{\pm} = \Lambda_{\pm} \free = \pm \sqrt{- \Delta +
	1} \, \Lambda_{\pm} = \pm \Lambda_{\pm} \sqrt{- \Delta + 1} \,
	\I_{4}.
\end{equation*}

The operator $|D_0| = \sqrt{-\Delta + 1} \, \I_{4}$ can be defined for
all $f \in H^{1}(\mathbb{R}^{3}, \mathbb{C}^{4})$ as the inverse
Fourier transform of the $L^{2}$ function $\sqrt{|p|^{2} + 1} \,
\I_{4}\, \hat{f} (p)$ (see \cite{LiebLoss}),

Now we consider the Foldy-Wouthuysen (FW) transformation, given by the
unitary transformation $\fw = \mathcal{F}^{-1} U(p) \mathcal{F}$.
Under the FW transformation the projectors $\Lambda_{\pm}$ become
simply
\begin{equation*}
	\Lambda_{\pm,_{ FW}} = \fw \Lambda_{\pm} \fw^{-1} = \frac{\I_{4}
	\pm \mathbf{\beta}}{2},
\end{equation*}
and   $D_{_\text{FW}} = \fw D_{0}\fw^{-1} =|D_0|  \, \beta$
with the corresponding sesquilinear form
\begin{equation*}
	\mathcal{Q}_{ D_{_\text{FW}}}(f,g) = \int_{\mathbb{R}^{3}}
	\sqrt{|p|^{2} + 1} \, \, ({\hat{f}} (p), \beta \hat{g}(p)) \, dp =
	\mathcal{Q}_{ D_0}(\fw^{-1}f,\fw^{-1}g)
\end{equation*}
defined on the form domain $H^{1/2}(\mathbb{R}^{3};\mathbb{C}^{4})$.

The operator $\sqrt{- \Delta + 1}$, exactly as the fractional
Laplacian, can be related to the following Dirichlet to Neumann
operator (see for example \cite{CabreMorales05} for problems involving
the fractional laplacian, and \cite{CZNolasco2011, CZNolasco2013} for
more closely related models): given $u$ solve the Dirichlet problem
\begin{equation*}
	\begin{cases}
		- \partial^{2}_{x}v - \Delta_{y} v + v = 0 & \text{in }
		\mathbb{R}^{4}_{+} = \settc{(x,y) \in \mathbb{R} \times
		\mathbb{R}^{3}}{x > 0 } \\
		v(0,y) = u(y)  &\text{for } y
		\in \mathbb{R}^{3} = \partial \mathbb{R}^{4}_{+}.
	\end{cases}
\end{equation*}
and let 
\begin{equation*}
	\mathcal{T} u (y) = \frac{\partial v}{\partial \nu}(0,y) = -
	\frac{\partial v}{\partial x}(0,y)
\end{equation*}
Then $\mathcal{T} u (y) = \mathcal{F}^{-1}_{y} (\sqrt{|p|^2 + 1} \,
\hat{u}(p )) = \sqrt{- \Delta + 1} \, u(y)$.

Indeed, solving the equation via partial Fourier transform we get
\begin{equation*}
	v (x,y) = \mathcal{F}^{-1}_{y} (\hat{u}(p) e^{- x \sqrt{|p|^2 + 1}
	} ).
\end{equation*}

In view of the FW transformation we may consider the eigenvalue problem
\eqref{eq:MCD} for the perturbed Dirac operator
\begin{equation*}
	D_{0} - \cfstr \dalfa \cdot A + \cfstr \Phi + \cfstr V
\end{equation*}
as follows.  

Let $(\psi_{\mult}, \mult) \in H^{1/2} (\mathbb{R}^{3},\mathbb{C}^{4})
\times \R$ be a (weak) solution of the eigenvalue problem
\eqref{eq:MCD} and let $\phi_{\mult} $ be the following extension of
$\varphi_{\mult} = \fw \psi_{\mult} $ on the half-space (see lemma
\ref{lem:betterminimizer} below)
\begin{equation}  
    \label{eq:extension} 
	\phi_{\mult}(x,y) = \mathcal{F}^{-1}_{y} (U(p)
	\hat{\psi}_{\mult}(p ) e^{- x \sqrt{ |p|^2 + 1}}),
\end{equation} 
then $\phi_{\mult} \in H^{1}(\R^4_{+},\mathbb{C}^4)$,
$(\phi_{\mult} )_{tr} = \varphi_{\mult} $ and $\phi_{\mult} $ is
a (weak) solution of the following Neumann boundary value problem
\begin{equation}
	\label{eq:neumannpde}
	\tag{$\mathcal{P}_{\mult}$}
	\begin{cases}
		-\partial^2_x \phi_{\mult} - \Delta_{y}
		\phi_{\mult} + \phi_{\mult} = 0 &\text{in }
		\mathbb{R}^{4}_{+} \\
		\beta \displaystyle{\frac{\partial
		\phi_{\mult}}{\partial \nu} } + \fw ( - \cfstr \dalfa
		\cdot A + \cfstr \Phi + \cfstr V)
		\fw^{-1}\varphi_{\mult}= \mult
		\varphi_{\mult} &\text{on } \partial\mathbb{R}^{4}_{+} =
		\mathbb{R}^{3} \\
		|\varphi_{\mult}|^2_{L^2} =1 \\
		\Phi = |\fw^{-1}\varphi_{\mult}|^2 \ast \frac{1}{|x|} ;
		\quad A = (\fw^{-1} \varphi_{\mult}, \dalfa
		\fw^{-1}\varphi_{\mult} ) \ast \frac{1}{|x|}
    \end{cases}
\end{equation}
On the other hand, if $\phi_{\mult} \in H^{1}(\R^4_{+},\mathbb{C}^4) $
is a (weak) solution of the Neumann boundary value problem
\eqref{eq:neumannpde}, setting $\varphi_{\mult} = (\phi_{\mult} )_{tr}
$, then $(\fw^{-1} \varphi_{\mult} , \mult) \in H^{1/2}
(\mathbb{R}^{3},\mathbb{C}^{4}) \times \R$ is a (weak) solution of
(MDC).

\section{Notation and preliminary results}
 
To simplify the notation when clear from the context we will denote
simply with $H^{1/2}$ the Sobolev space
$H^{1/2}(\mathbb{R}^{3},\mathbb{C}^{n})$, with $H^{1}$ the space
$H^{1}(\mathbb{R}^{4}_{+},\mathbb{C}^{n})$ and with $L^{2}$ the spaces
$L^{2}(\mathbb{R}^{3},\mathbb{C}^{n})$ and $L^{2}(\mathbb{R}^{4}_{+},
\mathbb{C}^{n})$ (where $n = 2$ or $n = 4$).

We introduce the following scalar products and norms in $H^{1}$,
$H^{1/2}$ and $L^{2}$, respectively,
\begin{align*}
	&\langle f | g \rangle_{H^1} = \iint_{\R^4_{+}} ((\partial_{x} f,
	\partial_{x} g) + (\nabla_{y} f, \nabla_{y} g) + (f,g) ), \qquad
	\norm{f}_{H^{1}}^{2} = \langle f, f \rangle_{H^{1}}, \\
	&\langle f | g \rangle_{H^{1/2}} = \int_{\R^3} \sqrt{|p|^2 + 1}
	(\hat{f}, \hat{g}), \qquad \abs{f}_{H^{1/2}}^{2} = \langle f, f
	\rangle_{H^{1/2}}, \\
	&\abs{f}_{L^{2}}^{2} = \int_{\R^3} \abs{f}^{2}, \qquad
	\norm{f}^{2}_{L^{2}} = \iint_{\R^4_{+}} \abs{f}^{2}
\end{align*}
where $(v,w)$ denotes the scalar product in $\mathbb{C}^{n}$.

The following property can be easily verified (see
\cite{CotiZelati_Nolasco_2016}).
\begin{lem}
    \label{lem:betterminimizer}
	For $w \in H^{1}(\mathbb{R}^{4}_{+})$, let $w_{\text{tr}} \in
	H^{1/2}(\mathbb{R}^{3})$ be the trace of $w$ and define
    \begin{equation*}
		v(x,y) = \mathcal{F}^{-1}_{y} (\hat{w}_{\text{tr}}(p ) e^{- x
		\sqrt{|p|^2 + 1}}).
    \end{equation*}
    
    Then $v \in H^{1}(\mathbb{R}^{4}_{+})$ and
    \begin{equation}
		\label{eq:normaTracciaNormaH1}
		|w_{\text{tr}} |^2_{H^{1/2}} = \| v \|^2_{H^1} \leq \| w
		\|^2_{H^1}
    \end{equation}
\end{lem}

\begin{rem}
    \label{rem:positivita_quadratica}
	We recall that for all $f \in C^{\infty}_{0}(\mathbb{R}^{4})$
    \begin{equation*}
		\int_{\mathbb{R}^{3}} |f(0,y)|^{2} dy = \int_{\mathbb{R}^{3}}
		dy \int_{+\infty}^{0} \partial_{x} |f|^{2} dx \leq 2 \| f
		\|_{{L^2}} \| \partial_{x} f \|_{{L^2}}
    \end{equation*}
	and by density we get for all $\phi \in H^1$
	\begin{equation}
		\label{eq:stimaNormaL2NormaH1}
		|\phi_{tr}|^{2}_{L^2} \leq \iint_{\R^4_{+}} ( |\partial_{x}
		\phi |^2 + |\phi |^2) \, dx dy \leq \norm{\phi}^{2}_{H^{1}}.
	\end{equation}
\end{rem}

\begin{rem}
	\label{rem:Tix_convolution}
	Let us recall the following Hardy-type inequalities :
	\begin{description}
		\item[Hardy] for all $\psi \in H^1(\R^3)$
		\begin{equation*}
			| |x|^{-1} \psi |_{{L^2}} \leq 2 | \nabla \psi |_{{L^2}}
			\leq \gamma_{H} | |D_0| \psi |_{{L^2}}
		\end{equation*}
		where $\gamma_H = 2$.  
		
		\item[Kato] for all $\psi \in H^{1/2}(\R^3)$
		\begin{equation}
			\label{eq:Kato}
			| |x|^{-\frac{1}{2}} \psi |^2_{L^2} \leq \frac{\pi}{2}
			|(-\Delta)^{1/4} \psi |^2_{L^2} \leq \gamma_{K} |
			\psi|^2_{H^{1/2}}
		\end{equation}
		where $\gamma_{K} = \frac{\pi}{2}$.
    
		\item[Tix \cite{Tix_1998}] for all $\psi \in H^{1/2}(\R^3,
		\mathbb{C}^4)$
		\begin{equation}
			\label{eq:Tix}
			| |x|^{-\frac{1}{2}} \Lambda_{\pm} \psi |^2_{L^2} \leq
			\gamma_{T} | \Lambda_{\pm} \psi |_{H^{1/2}}^{2}
		\end{equation}
		where $\gamma_{T}= \frac{1}{2}{(\frac{\pi}{2} +
		\frac{2}{\pi})}$.
	\end{description}
\end{rem}

In view of the above inequalities, since $\Lambda_{\pm}$ commute with
translation we have the following result

\begin{lem}
	\label{lem:key_estimate}
	For any $\rho \in L^{1}(\mathbb{R}^{3})$ and $\psi \in
	H^{1/2}(\R^3,\mathbb{C}^4) $ we have
	\begin{align}
		& \int_{\R^3 } (\rho \ast \frac{1}{|x|} ) | \psi |^2 (y) dy
		\leq \frac{\pi}{2} \abs{\rho}_{L^{1}} |(-\Delta)^{1/4} \psi
		|^2_{L^2} \leq \gamma_{K} \abs{\rho}_{L^{1}} | \psi
		|_{H^{1/2}}^{2} \label{eq:KatoConvoluta}\\
		& \int_{\R^3 } (\rho \ast \frac{1}{|x|} ) | \Lambda_{\pm}\psi
		|^2 (y) dy \leq \gamma_{T} \abs{\rho}_{L^{1}} | \Lambda_{\pm}
		\psi |_{H^{1/2}}^{2}.
		\label{eq:key_estimate}
	\end{align}
\end{lem}

\begin{proof}
	\begin{multline*}
		\int_{\R^3 } \left( \int_{\R^{3}} \frac{\rho(x)}{|x-y|} \, dx
		\right) | \psi |^2 (y) \, dy = \int_{\R^3 } \left(
		\int_{\R^{3}} \frac{| \psi |^2 (y)}{|x-y|} \, dy \right) 
		\rho(x) \, dx
		\\
		\leq \frac{\pi}{2} \abs{\rho}_{L^{1}} |(-\Delta)^{1/4} \psi
		|^2_{L^2} \leq \gamma_{K} \abs{\rho}_{L^{1}} | \psi
		|_{H^{1/2}}^{2}
	\end{multline*}
	The second inequality can be proved in the same way since
	$\Lambda_{\pm}$ commute with translations.
\end{proof}

Hence in particular for $V(x) = - \frac{Z}{|x|}$ and $Z \leq Z_{c}=
124$ we have that $Z \cfstr \gamma_T \in (0,1)$ and
\begin{equation}
	\label{eq:stimaPotenzialeTix}
	\cfstr \int \abs{V}\abs{\Lambda_{\pm} \psi}^{2} \, dy = \cfstr
	||V|^{1/2} \Lambda_{\pm} \psi |^{2}_{L^2} \, dy \leq Z \cfstr \gamma_T |
	\Lambda_{\pm} \psi |_{H^{1/2}}.
\end{equation}

We consider the smooth functional $\mathcal{I} \colon H^{1}(\R^4_{+},
\mathbb{C}^4) \to \R $ given by
\begin{align*}
	\mathcal{I} (\phi) = & \| \phi_{1} \|^2_{H^1} - \| \phi_{2}
	\|^2_{H^1} \ + \cfstr \int_{\R^3}V \rho_{\psi} \, dy \\
	& + \frac{\cfstr}{2} \iint_{\R^3 \times \R^3} \frac{ \rho_{\psi}
	(y) \rho_{\psi} (z) } {|y - z| } \, dy \, dz -
	\frac{\cfstr}{2} \iint_{\R^3 \times \R^3 } \frac{ J_{\psi}(y)
	\cdot J_{\psi} (z)} {|y - z|} \, dy \, dz
\end{align*}
where $ \phi = \left( \begin{smallmatrix} \phi_1 \\ \phi_2
\end{smallmatrix} \right) \in H^1( \R^4_+ ; \mathbb{C}^2 \times
\mathbb{C}^2) $, $\psi= \fw^{-1} \phi_{tr} $, and $\rho_{\psi} =
|\psi|^2$, $J_{\psi} = ( \psi , \dalfa \psi )$.

It is easy to check that $(\phi_{\mult} , \mult)\in
H^{1}(\R^4_{+},\mathbb{C}^4) \times \R $ is a weak solution of the
Neumann boundary value problem \eqref{eq:neumannpde} if and only if
\begin{equation*}
	d\mathcal{I} (\phi_{\mult}) [h]= \mult \, 2 \RE \,
	\left\langle (\phi_{\mult})_{tr} | h_{tr} \right\rangle_{{L^2}}
	\qquad \forall h \in H^{1}(\R^4_{+}, \mathbb{C}^4).
\end{equation*}      
where $ d \mathcal{I} (\phi) : H^1 \to \R$ is the Frech\'et derivative
of the functional $\mathcal{I}$ given by
\begin{align*}
	d \mathcal{I} (\phi) [h] = & 2 \RE \langle \phi_{1} | h_{1}
	\rangle_{H^1} - 2 \RE \langle \phi_{2} | h_{2} \rangle_{H^1} + 2
	\cfstr \int_{\R^3 } V \RE (\psi ,\xi ) \, dy \\
	& + 2 \cfstr \iint_{\R^3 \times \R^3 } \frac{ \rho_{\psi} (y) \RE
	(\psi, \xi )(z) - J_{\psi} (y) \cdot \RE (\psi, \dalfa
	\xi )(z)} { |y-z | } \, dy \, dz
\end{align*}
where $h= \left( \begin{smallmatrix} h_1 \\ h_2 \end{smallmatrix}
\right) \in H^1( \R^4_+ ; \mathbb{C}^2 \times \mathbb{C}^2) $ and $\xi
= \fw^{-1} h_{tr}$.

Let compute also $ d^2 \mathcal{I} (\phi) : H^1 \times H^1 \to \R$,
setting $\eta = \fw^{-1} k_{tr} $ we have
\begin{align*}
	d^2 \mathcal{I} &(\phi) [h; k] = 2 \RE \langle k_{1} | h_{1}
	\rangle_{H^1} - 2 \RE \langle k_{2} | h_{2} \rangle_{H^1} + 2
	\cfstr \int_{\R^3} V \RE (\eta ,\xi ) \, dy \\
	& + 2 \cfstr \iint_{\R^3 \times \R^3 } \frac{ \rho_{\psi} (y) \RE
	(\eta ,\xi ) (z) - \ J_{\psi}(y) \cdot \RE (\eta
	,\dalfa \xi ) (z)} { |y-z | } \, dy \, dz \\
	& + 4 \cfstr \iint_{\R^3 \times \R^3 } \frac{\RE (\psi, \eta )(y)
	\RE(\psi, \xi )(z) -\RE (\psi, \dalfa \eta )(y) \cdot
	\RE (\psi, \dalfa \xi )(z) } { |y-z | } \, dy \,
	dz
\end{align*}

\begin{rem}
	\label{rem:pos_current_term}
	Note that for any $f \in L^1 \cap L^{3/2}$ we have that (see
	\cite[Corollary 5.10]{LiebLoss})
	\begin{equation*}
		\int_{\R^3 \times \R^3} \frac{f(y) \bar{f}(z) }{|y-z|} =
		\sqrt{\frac{2}{\pi}} \int_{\R^3} \frac{1}{|p|^2}
		|\hat{f}|^2(p) \, dp \geq 0.
	\end{equation*} 
	Hence in particular
	\begin{equation}
		\label{eq:correntePositiva}
		\iint_{\R^3 \times \R^3 } \frac{ J_{\psi} (y) \cdot J_{\psi}
		(z) } { |y-z | } \, dy \, dz \geq 0 .
	\end{equation} 

	Moreover since $ |J_{\psi}(y) | \leq \rho_{\psi} (y) $ for any $y
	\in \R^3$ and $ \psi \in H^{1/2}$, see \cite[Lemma
	2.1]{Georgiev_Esteban_Sere_1996}, we have that
	\begin{equation}
		\label{eq:stimaGES}
		\iint_{\R^3 \times \R^3 } \frac{ \rho_{\psi} (y) \rho_{\psi}
		(z) - J_{\psi} (y) \cdot J_{\psi} (z) } { |y-z | } \, dy
		\, dz \geq 0.
	\end{equation}
\end{rem}

We also recall the following convergence result.  Let $ v \in
H^{1/2}$, $f_n, g_n, h_n$ bounded sequences in $H^{1/2}$, and one of
them converge weakly to zero in $H^{1/2}$, then we have (see for
example \cite[Lemma 4.1]{CZNolasco2013})
\begin{equation}
	\label{eq:convergenceto0}
	\iint_{\R^3 \times \R^3} \frac{ |f_n|(y) |g_n|(y) |v|(z ) |h_n|(z)
	}{|y-z|} \, dy dz \to 0.  \qquad \text{as} \, n \to + \infty
\end{equation}        

The following lemma is essentially already contained in \cite[Lemma
B.1]{CotiZelati_Nolasco_2016}, see also \cite{Morozov_Vugalter_2006}
for related results.

\begin{lem}
	\label{lem:stima_commutatori}      
	Let $\chi \in C_0^{\infty}(\R^3)$, then $ [ \chi ,\fw^{-1}] $ and
	$ [ \chi ,\fw] $ are bounded operator from $H^{1/2} (\R^3;
	\mathbb{C}^4) \to H^{3/2} (\R^3; \mathbb{C}^4)$
    
	Moreover, for $R \geq 1 $ let define $\chi_{R}(y) = \chi ( R^{-1}
	y)$.  Then
	\begin{equation*}
		\| [ \chi_{R},\fw] \|_{ H^{1/2} \to H^{1/2} }= \|[ \chi_{R}
		,\fw^{-1}] \| = O(R^{-1}) \qquad \text{as} \quad R \to +
		\infty.
	\end{equation*}
\end{lem}

\section{Maximization problem}

Our first step will be to maximize our functional in the sets
\begin{equation*}
	\mathcal{X}_{W} = \settc{ \phi =\left( \begin{smallmatrix} \phi_1 \\
	\phi_{2} \end{smallmatrix} \right) \in H^1(\R^4_{+}; \C^2 \times
	\C^2)}{\phi_1 \in W, \ |\phi_{tr}|^2_{L^2} = 1}.
\end{equation*}    
depending on a 1-dim vector space $W \subset H^{1}(\R^4_{+}; \C^2)$.
For each $\phi \in \mathcal{X}_{W}$ we will write $\phi_{2} \in
X_{-}$, so that $\phi \in W \times X_{-}$.

Denoting $\mathcal{G}(\phi) = \ |\phi_{tr}|^2_{L^2}$, the tangent
space of $ \mathcal{X}_{W} $ at some point $\phi \in \mathcal{X}_{W}$
is the set
\begin{equation*}
	T_{\phi} \mathcal{X}_{W} = \settc{h \in W \times
	X_{-}}{d\mathcal{G}(\phi)[h] \equiv 2 \RE \langle\phi_{tr} |
	h_{tr} \rangle_{_{L^{2}} }= 0}
\end{equation*} 
and $\nabla_{\mathcal{X}_{W} }\mathcal{I}(\phi)$, the projection of
the gradient $\nabla \mathcal{I}(\phi)$ on the tangent space
$T_{\phi}\mathcal{X}_{W} $ is given by
\begin{equation*}
	\nabla_{\mathcal{X}_{W} } \mathcal{I}(\phi) = \nabla
	\mathcal{I}(\phi) - \mult(\phi) \nabla \mathcal{G}(\phi)
\end{equation*}
where $\nabla \mathcal{I}(\phi)$, $\nabla \mathcal{G}(\phi) \in H^{1}$
are such that
\begin{equation*}
	\RE \langle \nabla \mathcal{I}(\phi) | h \rangle_{_{H^{1}}} =
	d\mathcal{I}(\phi)[h] \quad \text{ and } \quad \RE \langle \nabla
	\mathcal{G}(\phi) | h \rangle_{_{H^{1}}} = d\mathcal{G}(\phi)[h]
\end{equation*}
for all $ h \in H^{1}$ and $\mult(\phi) \in \mathbb{R}$ is such that
$\nabla_{\mathcal{X}_{W} }\mathcal{I}(\phi) \in T_{\phi}
\mathcal{X}_{W}$.

We begin giving a result on Palais-Smale sequences for $\mathcal{I}$ 
restricted on $\mathcal{X}_{W}$.
\begin{lem}
	\label{lem:PalaisSmale}
	Fix any $w \in H^{1}(\mathbb{R}^{4}_{+};\mathbb{C}^{2})$,
	$(w)_{\text{tr}} \neq 0$ and let $W = \spann\{w\}$.
	
	Suppose $\phi^{n} \in \mathcal{X}_{W}$ is a Palais-Smale sequence
	for $\mathcal{I}$ restricted on $\mathcal{X}_{W}$, at a positive
	level, that is
	\begin{itemize}
		\item $\mathcal{I}(\phi^{n}) = c + \epsilon_{n} \to c > 0$;
	
		\item $\nabla_{\mathcal{X}_{W}}\mathcal{I}(\phi^{n}) \to 0$.
	\end{itemize}
	
	Then $\phi^{n}$ is bounded and $| (\phi_{1}^{n})_{tr} |^2_{L^2} >
	\frac{1}{2}$.
\end{lem}

\begin{proof}
	We let $\phi^{n} = \left( \begin{smallmatrix} \phi^{n}_{1} \\
	\phi^{n}_{2} \end{smallmatrix} \right)$.  Since $\phi_{1}^{n} \in
	W$, $W$ one dimensional, and $0 < |(\phi_{1}^{n})_{tr} |^2_{L^2}
	\leq 1$ we have $\norm{\phi_{1}^{n}} \leq c_{W}$ for some constant
	(depending on $W$).
	
	Let us denote $\psi_{+}^{n} = \fw^{-1} \left(
	\begin{smallmatrix} (\phi_{1}^{n})_{tr} \\ 0 \end{smallmatrix}
	\right)$, $\psi_{-}^{n} = \fw^{-1} \left(
	\begin{smallmatrix} 0 \\ (\phi_{2}^{n})_{tr} \end{smallmatrix}
	\right)$ and $\psi^{n} = \psi_{+}^{n} + \psi_{-}^{n}$.  In view of
	Remarks \ref{rem:Tix_convolution} and \ref{rem:pos_current_term}
	we have, for $n$ large enough,
	\begin{align*} 
		c + \epsilon_{n} &= \mathcal{I}(\phi^{n}) \leq \|\phi_{1}^{n}
		\|^2_{H^1} - \|\phi_{2}^{n} \|^2_{H^1}
		+ \cfstr \int_{R^3 \times \R^3} \frac{\rho_{\psi^{n}}
		(y) ( \rho_{\psi_{+}^{n}} + \rho_{\psi_{-}^{n}})(z)}{|y-z|} \,
		dy dz\\
		&\leq (1 + \cfstr\gamma_{T}) \|\phi_{1}^{n} \|^2_{H^1} - (1 -
		\cfstr\gamma_{T}) \|\phi_{2}^{n} \|^2_{H^1}
	\end{align*}
	Hence we may conclude that
	\begin{equation*}
		\|\phi_{1}^{n} \|^2_{H^1} \leq c_{W}, \qquad \|\phi_{2}^{n}
		\|^2_{H^1} \leq \frac{1 + \cfstr\gamma_{T}}{1 -
		\cfstr\gamma_{T}} \|\phi_{1}^{n} \|^2_{H^1}
	\end{equation*}
	and also
	\begin{equation*}
		\|\phi_{1}^{n} \|^2_{H^1}+ \|\phi_{2}^{n} \|^2_{H^1} \leq
		\frac{2c_{W}}{1 - \cfstr\gamma_{T}}
	\end{equation*}
	In particular we deduce that the any Palais-Smale sequence is
	bounded in $H^{1}$.
	
	Then we have
	\begin{align*}
		\langle \nabla_{\mathcal{X}_{W}} \mathcal{I}(\phi^{n}),
		\phi^{n} \rangle &= d\mathcal{I}(\phi^{n})[\phi^{n}] -
		\mult(\phi^{n}) 2 |(\phi^{n})_{tr} |^2_{{L^{2}} } \\
		&= 2 \mathcal{I}(\phi^{n})  - 2 \mult(\phi^{n}) \\
		& \qquad + \cfstr \iint_{\R^3 \times \R^3}
		\frac{\rho_{\psi^{n}}(y) \rho_{\psi^{n}}(z) - J_{\psi^{n}}(y)
		\cdot J_{\psi^{n}}(z)}{|y-z|} \, dy dz
	\end{align*}
	and we deduce that
	\begin{multline}
		\label{eq:moltiplicatore}
		\mult(\phi^{n}) = c + \epsilon_{n} + \langle
		\nabla_{\mathcal{X}_{W}} \mathcal{I}(\phi^{n}), \phi^{n}
		\rangle \\
		+ \frac{\cfstr}{2} \iint_{\R^3 \times \R^3}
		\frac{\rho_{\psi^{n}}(y)\rho_{\psi^{n}}(z) - J_{\psi^{n}}(y)
		\cdot J_{\psi^{n}}(z)}{|y-z|} \, dy dz.
	\end{multline}
	and
	\begin{equation}
		\mult(\phi^{n}) > 0
	\end{equation}
	for $n$ large enough since the last term is non negative and
	$\langle \nabla_{\mathcal{X}_{W}} \mathcal{I}(\phi^{n}), \phi^{n}
	\rangle \to 0$.

	Moreover since $\norm{\phi^{n}}_{H^{1}}$ is bounded we have
	\begin{equation*}
		o(1)= d\mathcal{I}(\phi^{n})[\beta \phi^{n}] - \mult(\phi^{n})
		2 \RE \langle (\phi^{n})_{tr} |
		(\beta\phi^{n})_{tr}\rangle_{{L^{2}}},
	\end{equation*}
	and observing that
	\begin{equation*}
		\RE (\psi_{+}(y) + \psi_{-}(y), \psi_{+}(y) - \psi_{-}(y)) =
		\abs{\psi_{+}(y)}^{2} - \abs{\psi_{-}(y)}^{2} =
		\rho_{\psi_{+}}(y) - \rho_{\psi_{-}}(y),
	\end{equation*}
	we deduce that
	\begin{align*}
		\mult(\phi^{n}) |\psi_{+}^{n} |^2_{L^2} &+ o(1) =
		\mult(\phi^{n}) |\psi_{-}^{n} |^2_{L^2} + \frac{1}{2}
		d\mathcal{I}(\phi^{n})[\beta \phi^{n}] \\
		&= \mult(\phi^{n}) |\psi_{-}^{n} |^2_{L^2} + \|\phi_{1}^{n}
		\|^2_{H^1} + \|\phi_{2}^{n} \|^2_{H^1} + \cfstr \int_{\R^3} V
		\rho_{\psi_{+}}^{n} \, dy \\
		&\quad - \cfstr \int_{\R^3} V \rho_{\psi_{-}^{n}} \, dy \\
		&\quad + \cfstr \iint_{\R^3 \times \R^3}
		\frac{\rho_{\psi^{n}}(y) ( \rho_{\psi_{+}^{n}} -
		\rho_{\psi_{-}}^{n})(z)}{ |y-z|} \, dy dz \\
		&\quad - \cfstr \iint_{\R^3 \times \R^3} \frac{J_{\psi^{n}}(y)
		\cdot( J_{\psi_{+}^{n}} - J_{\psi_{-}^{n}}) (z)}{ |y-z|} \, dy
		dz \\
		&\geq \mult(\phi^{n}) |\psi_{-}^{n} |^2_{L^2} + (1 -
		Z \cfstr \gamma_T) \|\phi_{1}^{n} \|^2_{H^1} + \|\phi_{2}^{n}
		\|^2_{H^1} \\
		&\quad - \cfstr \iint_{\R^3 \times \R^3}
		\frac{(\rho_{\psi^{n}} + |J_{\psi^{n}}|)(y)
		\rho_{\psi_{-}^{n}}(z)}{ |y-z|} \, dy dz \\
		&\quad + \cfstr \iint_{\R^3 \times \R^3}
		\frac{(\rho_{\psi^{n}}- |J_{\psi^{n}}|)(y)
		\rho_{\psi_{+}^{n}}(z)}{ |y-z|} \, dy dz \\
		&\geq \mult(\phi^{n}) |\psi_{-}^{n} |^2_{L^2} + (1 -
		Z \cfstr \gamma_T) \|\phi_{1}^{n} \|^2_{H^1} + (1 - 2\cfstr
		\gamma_{T}) \|\phi_{2}^{n} \|^2_{H^1}\\
		&> \mult(\phi^{n}) |\psi_{-}^{n} |^2_{L^2},
	\end{align*}
	where we have used the estimate \eqref{eq:key_estimate}.  We
	immediately deduce, since $\mult(\phi^{n}) > 0$ for $n$ large
	enough, that $|\psi_{+}^{n} |^2_{L^2} > |\psi_{-}^{n} |^2_{L^2}$
	which implies that $|\psi_{+}^{n} |^2_{L^2} > \frac{1}{2}$.
\end{proof}

We now introduce the maximization problem
\begin{equation} 
	\label{eq:maximum_pb}
	\lambda_{W} = \sup_{ \phi \in \mathcal{X}_{W} } \mathcal{I}(\phi),
\end{equation}
and we show that $\lambda_{W}$ is positive.
\begin{lem}
	\label{lem:apriori_estimates}
	Fix any $w \in H^{1}(\R^4_{+}; \mathbb{C}^2)$ and let $ W = \spann
	\{ w \} $.  If $w_{tr} \equiv 0 $ then $ \sup_{ \phi \in
	\mathcal{X}_{W} } \mathcal{I}(\phi) = + \infty $; on the other
	hand for $w_{tr} \not\equiv 0 $ then
	\begin{equation}
		\label{eq:p1}
		\sup_{ \phi \in \mathcal{X}_{W} } \mathcal{I}(\phi) =
		\lambda_{_{W}} \in ( 0, + \infty) .
	\end{equation}
\end{lem}
	
\begin{proof}
	If $w_{tr} \equiv 0 $ we take a sequence $\phi_n = \left(
	\begin{smallmatrix} a_n w\\
	\phi_2\end{smallmatrix} \right) \in \mathcal{X}_{W} $ with $ |a_n|
	\to + \infty$, for $n \to + \infty$, and a fixed $\phi_2 \in H^1$
	such that $|(\phi_2)_{tr} |^2_{L^2}= 1$.  We denote $\psi_{-} =
	\fw^{-1} \left( \begin{smallmatrix} 0\\
	(\phi_2)_{tr}\end{smallmatrix} \right)$.  Then by
	\eqref{eq:stimaGES} we have
	\begin{align*}
		\sup_{ \phi \in \mathcal{X}_{W} } \mathcal{I}(\phi) & \geq
		\mathcal{I}(\phi_n) \geq |a_n|^2 \|w\|^2_{H^1} - \|\phi_2
		\|^2_{H^1} + \cfstr \int_{\R^3} V \rho_{\psi_{-}}(y) \, dy \\
		& \geq |a_n|^2 \|w\|^2_{H^1} - C \to + \infty \qquad \text{as
		} \,\, n \to + \infty.
	\end{align*}
	for some constant $C>0$ independent on $n \in \N$.

	Fix now $w \in H^{1}(\R^4_{+}, \mathbb{C}^2)$ with $|w_{tr}|_{L^2}
	=1$.  Denote $ W = \spann \{\left( \begin{smallmatrix} w\\
	0\end{smallmatrix} \right) \} $, then $\phi = \left(
	\begin{smallmatrix} \phi_1 \\ \phi_2 \end{smallmatrix} \right) \in
	\mathcal{X}_{W} $ is given by $ \phi_1 = a w$ , $a \in \C $ and $
	|\phi_{tr}|_{L^2}^2= |a|^2 + |(\phi_{2})_{tr} |_{L^2}^2 = 1$.
	Denote $v_{+} = \fw^{-1} \left( \begin{smallmatrix} w_{tr} \\
	0\end{smallmatrix} \right) $, $\psi_{+} = a v_{+}$, $\psi_{-} =
	\fw^{-1} \left( \begin{smallmatrix} 0 \\
	(\phi_2)_{tr}\end{smallmatrix} \right) $ and $ \psi= \psi_{+}+
	\psi_{-}$.

	Since $ \lambda_{W} = \sup_{ \phi \in \mathcal{X}_{W} }
	\mathcal{I}(\phi) \geq \mathcal{I} (\left( \begin{smallmatrix} w\\
	0\end{smallmatrix} \right)) $, by \eqref{eq:stimaGES},
	\eqref{eq:stimaPotenzialeTix}, \eqref{eq:normaTracciaNormaH1} and
	\eqref{eq:stimaNormaL2NormaH1}
    \begin{multline*}
		\mathcal{I} (\left( \begin{smallmatrix} w\\ 0\end{smallmatrix}
		\right)) \geq \|w \|^2_{H^1} + \cfstr \int_{\R^3} V
		\rho_{v_{+}} \, dy \geq \|w \|^2_{H^1} - Z \cfstr \gamma_T
		\abs{\fw^{-1}w_{\text{tr}}}^2_{H^{1/2}} \\
		= \|w \|^2_{H^1} - Z \cfstr \gamma_T \abs{w_{\text{tr}}}^2_{H^{1/2}}
		\geq (1-Z \cfstr \gamma_T) \|w \|^2_{H^1} \geq (1- Z \cfstr \gamma_T).
	\end{multline*}
	hence $ \lambda_{W} > 0$.
	
	Moreover, in view of \eqref{eq:correntePositiva},
	\eqref{eq:key_estimate} and recalling that
	$\abs{\psi_{-}}_{H^{1/2}} = \abs{(\phi_{2})_{\text{tr}}}_{H^{1/2}}
	\leq \norm{\phi_{2}}_{H^{1}}$ by Lemma \ref{lem:betterminimizer}
	and that $\abs{\rho_{\psi}}_{L^{1}} = \abs{\psi}_{L^{2}} =
	\abs{\phi_{\text{tr}}}_{L^{2}} = 1$, for any $\phi \in
	\mathcal{X}_{W} $ we have
	\begin{align*}
		\mathcal{I} (\phi) \leq & \| \phi_1 \|^2_{H^1} - \|\phi_2
		\|^2_{H^1} + \frac{\cfstr}{2} \iint_{\R^3 \times \R^3 } \frac{
		\rho_{\psi} (y) \rho_{\psi}(z) } { |y-z | } \, dy \, dz\\
		\leq & |a|^2 \| w \|^2_{H^1}- \|\phi_2 \|^2_{H^1} + \cfstr
		\iint_{\R^3 \times \R^3 } \frac{ \rho_{\psi} (y) (
		\rho_{\psi_{-}} + \rho_{\psi_{+}} ) (z) } { |y-z | } \, dy \,
		dz\\
		\leq & |a|^2 \| w \|^2_{H^1}- \|\phi_2 \|^2_{H^1} +
		\cfstr\gamma_{T} |\rho_{\psi} |_{L^1} ( |\psi_{-}
		|^2_{H^{1/2}} + |a|^2 | v_{+} |^2_{H^{1/2}} ) \\
		\leq &- (1 - \cfstr\gamma_{T}) \|\phi_2 \|^2_{H^1} + C_{W}
	\end{align*}	
	hence in particular $\sup_{ \phi \in \mathcal{X}_{W} }
	\mathcal{I}(\phi) \leq C_{W}$ for some constant $C_{W} >0$
	depending only on $W$.
\end{proof}             

Fix now $w \in H^{1}(\R^4_{+}, \mathbb{C}^2)$ with $|w_{tr}|_{L^2}
=1$, to obtain additional information on the the maximization problem
\eqref{eq:maximum_pb} we introduce the (constraint) functional
$\mathcal{J}_{W} \colon B_{1} \to \R$ given by
\begin{equation*}
	\mathcal{J}_{W}(u) = \mathcal{I} \left(\left( \begin{smallmatrix}
	a ( u) w\\ u \end{smallmatrix} \right) \right)
\end{equation*}
where 
\begin{equation*}
	B_{1} = \settc{u \in X_{-} = H^{1}(\R^4_{+},
	\mathbb{C}^2)}{\abs{u_{\text{tr}}}_{L^{2}} < 1}
\end{equation*}
$a ( u) $ is given by the constrain equation $ |a w_{tr}
|^2_{L^2} + |u_{tr} |^2_{L^2} = 1 $ that is $ |a|^2 = 1 - |u_{tr}
|^2_{L^2} $.  By the phase invariance, without loss of generality, we
can always assume that $ a ( u) = \sqrt{ 1 - |u_{tr} |^2_{L^2}}$.
   
We have for any $h \in H^{1}(\R^4_{+}, \mathbb{C}^2)$,
\begin{equation*}
	d \mathcal{J}_{W}(u)[h] = d \mathcal{I} \left(\left(
	\begin{smallmatrix} a ( u) w\\ u \end{smallmatrix} \right) \right)
	\left[ \left(\begin{smallmatrix} da ( u) [h] w\\ h
	\end{smallmatrix} \right)\right]
\end{equation*}
and for $h, k \in H^1$
\begin{align}
	d^2 \mathcal{J}_{W}(u)[h ; k] &= Q_{1}[h; k] + Q_{2}[h; k]
	\nonumber \\
	&=d \mathcal{I} \left( \left( 
	\begin{smallmatrix} 
		a ( u) w\\ u
	\end{smallmatrix} 
	\right) \right) \left[ \left(
	\begin{smallmatrix}
		d^2 a ( u) [h;k] w\\ 0
	\end{smallmatrix} \right)\right] \label{eq:Q1}\\
	& \qquad + d^2 \mathcal{I} \left(\left( 
	\begin{smallmatrix} 
		a ( u) w\\ u 
	\end{smallmatrix} 
	\right) \right) \left[ \left(
	\begin{smallmatrix} 
		da ( u) [h] w\\ h 
	\end{smallmatrix}
	\right) ; \left( 
	\begin{smallmatrix} 
		da ( u) [k] w\\ k
	\end{smallmatrix} 
	\right) \right]\label{eq:Q2}
\end{align}
where, setting $\eta = \fw^{-1} \left( \begin{smallmatrix} 0 \\
u_{tr}\end{smallmatrix} \right) $, $\xi= \fw^{-1} \left(
\begin{smallmatrix} 0 \\ h_{tr} \end{smallmatrix} \right)$, we have
$a(u) = \sqrt{ 1 - |\eta |^2_{L^2}}$,
\begin{equation*}
	da (u) [h] = - \frac{\RE \langle \eta| \xi \rangle_{L^2}} { \sqrt{
	1 - |\eta |^2_{L^2}}} = - a(u) \frac{\RE \langle \eta | \xi
	\rangle_{L^2}} { 1 - |\eta |^2_{L^2}}
\end{equation*} 
and
\begin{equation*}
	d^2a (u) [h; h]= - a(u) \frac{ |\xi |^2_{L^2}} {1 - |\eta
	|^2_{L^2}} - a(u) \left( \frac{\RE \langle \eta |\xi
	\rangle_{L^2}} { 1 - |\eta |^2_{L^2}} \right)^2
\end{equation*} 
Setting $v_+ = \fw^{-1} \left( \begin{smallmatrix} w_{tr} \\ 0
\end{smallmatrix} \right) $, $\phi = \left( \begin{smallmatrix} a ( u)
w\\ u \end{smallmatrix} \right) $ and $\psi = \fw^{-1} \phi$ we have
\begin{align*}
	\frac{1}{2} &d \mathcal{J}_{W}(u)[ h] = - \RE \langle \eta | \xi
	\rangle_{L^2} \| w \|^2_{H^1} - \RE \langle u |h \rangle_{H^1} \\
	&\quad - \cfstr \RE \langle \eta | \xi \rangle_{L^2} \int_{\R^3} V
	\rho_{ v_+} \, dy + \cfstr \int_{\R^3} V \RE (\eta, \xi) \, dy \\
	&\quad + \cfstr \int_{\R^3} V \RE(a(u) v_+, \xi) \, dy - \cfstr
	\frac{ \RE \langle \eta | \xi \rangle_{L^2} } {1 - |\eta
	|^2_{L^2}} \int_{\R^3} V \RE(a(u)v_+, \eta)\, dy \\
	&\quad - \cfstr \RE \langle \eta | \xi \rangle_{L^2} \iint_{\R^3
	\times \R^3} \frac{\rho_{\psi} (y) \rho_{v_+}(z) - J_{\psi} (y)
	\cdot J_{v_+}(z) }{|y-z|} \, dy dz \\
	&\quad + \cfstr \iint_{\R^6} \frac{\rho_{\psi} (y)
	\RE(\eta, \xi) (z) - J_{\psi} (y) \cdot (\eta, \dalfa
	\xi)(z) }{ |y-z|} \, dy dz \\
	&\quad +\cfstr \iint_{\R^6} \frac{\rho_{\psi} (y) \RE( a(u)v_+,
	\xi)(z) - J_{\psi} (y) \cdot \RE(a(u)v_+, \dalfa \xi)(z) }{ |y-z|}
	\, dy dz \\
	&\quad - \cfstr \frac{\RE \langle \eta | \xi \rangle_{L^2} } { 1 -
	|\eta |^2_{L^2}} \iint_{\R^6} \frac{\rho_{\psi} (y)
	\RE(a(u)v_+, \eta)(z) - J_{\psi} (y) \cdot \RE(a(u) v_+,
	\dalfa \eta)(z) }{|y-z|} \, dy dz
\end{align*}
It is convenient to define, for any $\nu \in H^{1/2}$,
\begin{equation}
	\label{eq:Gamma}
	\begin{split}
		\Gamma_{\psi}& (\nu) = \cfstr \int_{\R^3} V \RE(a(u) v_+,
		\nu) \, dy \\
		& + \cfstr \iint_{\R^3 \times \R^3} \frac{\rho_{\psi} (y)
		\RE(a(u)v_+, \nu)(z) - J_{\psi} (y) \cdot \RE(a(u) v_+, \dalfa
		\nu)(z) }{|y-z|} \, dy dz
	\end{split}
\end{equation}
where $\psi = a(u) v_+ + \eta$.  Remark that we have
\begin{equation*}
	\abs{\Gamma_{\psi}(\nu)} \leq C \abs{\psi}_{H^{1/2}}
	\abs{\nu}_{H^{1/2}}
\end{equation*}
and
\begin{align*}
	\frac{1}{2} d \mathcal{J}_{W}(u)[ h] & = - \RE \langle \eta | \xi
	\rangle_{L^2} \| w \|^2_{H^1} - \RE \langle u |h \rangle_{H^1} \\
	&\qquad - \cfstr \RE \langle \eta | \xi \rangle_{L^2} \int_{\R^3}
	V \rho_{ v_+} \, dy + \cfstr \int_{\R^3} V \RE (\eta, \xi) \, dy
	\\
	&\qquad - \cfstr \RE \langle \eta | \xi \rangle_{L^2} \iint_{\R^3
	\times \R^3} \frac{\rho_{\psi} (y) \rho_{v_+}(z) - J_{\psi} (y)
	\cdot J_{v_+}(z) }{|y-z|} \, dy dz \\
	&\qquad + \cfstr \iint_{\R^3 \times \R^3} \frac{\rho_{\psi} (y)
	\RE(\eta, \xi) (z) - J_{\psi} (y) \cdot (\eta, \dalfa
	\xi)(z) }{|y-z|} \, dy dz \\
	&\qquad + \Gamma_{\psi} (\xi) - \frac{ \RE \langle \eta | \xi
	\rangle_{L^2} } { 1 - |\eta |^2_{L^2}} \Gamma_{\psi} (\eta)
\end{align*}
       
In particular, for $h = u$, that is $ \xi = \eta$, we have
\begin{align*}
	\frac{1}{2} d \mathcal{J}_{W}(u) [ u ] = & - |\eta|^2_{L^2} \| w
	\|^2_{H^1} - \|u \|^2_{H^1} - \cfstr |\eta|^2_{L^2} \int_{\R^3} V
	\rho_{ v_+} \, dy + \cfstr \int_{\R^3} V \rho_{\eta} \, dy \\
	&\qquad - \cfstr |\eta|^2_{L^2} \iint_{\R^3 \times \R^3}
	\frac{\rho_{\psi} (y) \rho_{v_+}(z) - J_{\psi} (y) \cdot
	J_{v_+}(z) }{|y-z|} \, dy dz \\
	&\qquad + \cfstr \iint_{\R^3 \times \R^3} \frac{\rho_{\psi} (y)
	\rho_{\eta}(z) - J_{\psi} (y) \cdot J_{\eta}(z) }{|y-z|} \, dy
	dz \\
	&\qquad + \left(1 - \frac{ |\eta|^2_{L^2}} {1- |\eta|^2_{L^2}}
	\right) \Gamma_{\psi} (\eta)
\end{align*}
hence we get
\begin{equation}
	\label{eq:stima_grad}
	\frac{1}{2} d \mathcal{J}_{W}(u)[ u] \leq - |\eta|^2_{L^2} (1-
	Z \cfstr \gamma_T) \| w \|^2_{H^1} - (1 - 2\cfstr \gamma_{T})\|u
	\|^2_{H^1} + \frac{1- 2 |\eta|^2_{L^2}} {1- |\eta|^2_{L^2}}
	\Gamma_{\psi} (\eta).
\end{equation}
and in particular for all $\abs{\eta}_{L^{2}} = \frac{1}{2}$ we have
\begin{equation}
	\label{eq:stima_gradPalla}
	\frac{1}{2} d \mathcal{J}_{W}(u)[ u] \leq - |\eta|^2_{L^2} (1-
	Z \cfstr \gamma_T) \| w \|^2_{H^1} - (1 - 2\cfstr \gamma_{T})\|u
	\|^2_{H^1} < 0
\end{equation}

\begin{prop}
	\label{prop:PalaisSmale2}
	Let $\{u_{n} \} \subset B_{1} \subset H^{1}(\R^4_{+},
	\mathbb{C}^2)$ be a Palais Smale sequence for $\mathcal{J}_{W}$,
	i.e. such that $ \| d\mathcal{J}_{W}(u_n) \| \to 0$ and
	$\mathcal{J}_{W}(u_n) \to c > 0$.
	
	Then,
	\begin{itemize}
		\item[(i)] there exists $\kappa >0 $ such that
		$|\eta_n |^2_{L^2} \leq \frac{1}{2} - \kappa$ for 
		all $n$ large enough;
                              
		\item[(ii)] $\{ u_n \}$ is precompact in $H^1$.
	\end{itemize} 
	where $\eta_{n} = (u_{n})_{\text{tr}}$.
\end{prop}

\begin{proof}
	It is clear that the sequence $\phi^{n} = \left( 
	\begin{smallmatrix} a(u_{n})w \\ u_{n} \end{smallmatrix}
	\right)$ is a Palais-Smale sequence for $\mathcal{I}$ restricted
	to the subspace $\mathcal{X}_{W}$.  We can then apply Lemma
	\ref{lem:PalaisSmale} to deduce
	that $\{ u_n \}$ is a bounded sequence in $H^1$ and
	$\abs{\eta_{n}} <\frac{1}{2}$.  We can assume that $u_n
	\rightharpoonup u$ weakly in $H^1$.
 
	We let $v_{+} = \fw^{-1} \left( \begin{smallmatrix} w_{tr} \\
	0\end{smallmatrix} \right) $,
	$a_n= a(u_{n}) \to a $ (up to subsequences) and $\psi_n = \fw^{-1}
	(\phi_n)_{tr} = a_n v_+ + \eta_{n}$.

	(i) Suppose on the contrary that $|\eta_n |^2_{L^2} \to
	\frac{1}{2}$.  Then from \eqref{eq:stima_grad}, $| \Gamma_{\psi_n}
	(\eta_n) | \leq C \abs{\psi_{n}}_{H^{1/2}} |\eta_n |_{H^{1/2}}$
	and the fact that $\{ u_n \}$ is a bounded sequence in $H^1$ , we
	get
	\begin{align*}
		\frac{1}{2} d \mathcal{J}_{W}(u_n)[ u_n] \leq & -
		|\eta_n|^2_{L^2} (1- Z \cfstr \gamma_T) \| w \|^2_{H^1} - (1 -
		\cfstr\gamma_{T}) \|u_n \|^2_{H^1} \\
		&\qquad + \frac{1- 2 |\eta_n|^2_{L^2}} {1- |\eta_n|^2_{L^2}}
		\Gamma_{\psi_n} (\eta_n) \\
		\leq & - |\eta_n |^2_{L^2}(1- Z \cfstr \gamma_T) \| w \|^2_{H^1} +
		\frac{ 1- 2|\eta_n |^2_{L^2}} {1- |\eta_n |^2_{L^2}} C
		\|u_n\|_{H^{1}} \\
		\leq & - \frac{1}{2} (1- Z \cfstr \gamma_T) \| w \|^2_{H^1} + o(1)
	\end{align*}
	a contradiction.
	 
	(ii) Since $ |\eta_n|^2_{L^2} \leq \frac{1}{2} - \kappa$ and
	$\{ u_n \}$ is a bounded sequence in $H^1$, by
	\eqref{eq:stima_grad} we may conclude that
	\begin{equation*}
		\Gamma_{\psi_n} (\eta_n) \geq - C \|d \mathcal{J}_{W}(u_n)\| =
		o(1)
	\end{equation*}
	for some constant $C>0$ independent on $n$.
     
	Now, from $ u_n \rightharpoonup u $ weakly in $H^1$, $\eta_n
	\rightharpoonup \eta = \fw^{-1} \left( \begin{smallmatrix} 0 \\
	u_{tr}\end{smallmatrix} \right)$ weakly in $H^{1/2}$ and $a_n \to
	a$ in $\C$, since $V v_{+ } \in H^{- 1/2}$ we have
	\begin{equation*}
		\int_{\R^3} V \RE(a_n v_+, (\eta_n - \eta))(y) \, dy \to 0
	\end{equation*}
	and in view of equation \eqref{eq:convergenceto0}
	\begin{align*}
		& \iint_{\R^3 \times \R^3} \frac{\rho_{\psi_n} (y) \RE(a_n
		v_+, (\eta_n - \eta))(z) }{|y-z|} \, dy \, dz\to 0 \\
		& \iint_{\R^3 \times \R^3} \frac{ J_{\psi_n} (y) \cdot \RE(a_n
		v_+, \dalfa( \eta_n - \eta))(z) }{|y-z|} \, dy \, dz \to 0
	\end{align*}
	then  we have
	\begin{align*}
		\frac{1}{2} d \mathcal{J}_{W} &(u_n) [u_n - u] \leq \\
		& \leq - | \eta_n -\eta |^2_{L^2} (1 - Z \cfstr \gamma_T) \| w \|^2_{H^1}
		- (1 - 2 \cfstr \gamma_{T}) \| u_n - u \|^2_{H^1} \\
		& \qquad - \frac{| \eta_n -\eta |^2_{L^2}} { 1 - |\eta_n
		|^2_{L^2}} \Gamma_{\psi_n}(\eta_n) + o(1) \\
		& \leq - | \eta_n -\eta |^2_{L^2} (1 - Z \cfstr \gamma_T) \| w
		\|^2_{H^1} - (1 - 2 \cfstr \gamma_{T}) \| u_n - u \|^2_{H^1} +
		o(1)
	\end{align*}
	Hence we may conclude that $u_n \to u $ strongly in $H^1$.
\end{proof}

We have the following strict concavity  result
        
\begin{prop}
	\label{prop:concavity}
	Let $u \in H^1$ be a critical point of $ \mathcal{J}_{W}$, namely
	$ d \mathcal{J}_{W}(u)[h] = 0$ for any $h \in H^1$, such that
	$|u_{tr}|^2_{L^2} < \frac{1}{2}$.
	 
	Then $u$ is a strict local maximum for $ \mathcal{J}_{W}$, namely
	\begin{equation*}
		d^2 \mathcal{J}_{W}(u)[h ; h] \leq - \delta \|h \|^2_{H^1}
		\qquad \forall \, h \in H^1
	\end{equation*}
	for some $\delta >0$.
\end{prop}
        
\begin{proof}
	Let $\phi = \left( \begin{smallmatrix} a(u) w\\ u
	\end{smallmatrix} \right) $ and $\psi = \fw^{-1} \phi_{tr} = a(u)
	v_{+} + \eta$ where $v_{+} = \fw^{-1} \left( \begin{smallmatrix}
	w_{tr} \\ 0\end{smallmatrix} \right)$ and $\eta = \fw^{-1} \left(
	\begin{smallmatrix} 0 \\ u_{tr}\end{smallmatrix} \right)$.  From
	the assumptions follows that $|\eta |^2_{L^2} < \frac{1}{2}$.
       
	Now, let $ d^2 \mathcal{J}_{W}(u)[h ; h] = Q_1 [h ; h] + Q_2 [h ;
	h] $ (see \eqref{eq:Q1}-\eqref{eq:Q2}).  We set $\xi= \fw^{-1}
	\left( \begin{smallmatrix} 0 \\ h_{tr} \end{smallmatrix} \right)$
	and
	\begin{align*}
		&r[\xi] = \frac{\RE \langle \eta, \xi \rangle_{L^2}} { 1 -
		|\eta |^2_{L^2}},\\
		&p[\xi; \xi] = (r[\xi] )^2 \geq 0, \\
		&q[\xi; \xi] = \frac{ |\xi |^2_{L^2}} {1 - | \eta |^2_{L^2}} +
		p[\xi; \xi] \geq 2 p[\xi; \xi].
	\end{align*}

	We have $ da(u) [\xi] = - a(u) r[\xi]$ and
	\begin{equation*}
		(da(u) [\xi] )^2 = a(u)^2 p[\xi; \xi] \geq 0 ; \qquad \quad d^2
		a(u) [\xi; \xi] = - a(u) q[\xi; \xi] \leq 0.
	\end{equation*}
	Since $ d \mathcal{J}_{W}(u)[ u] = 0$, in view of
	\eqref{eq:stima_grad} we have $\Gamma_{\psi}(\eta) \geq 0$.

	Let us compute $Q_1 [h ; h]$ adding a zero term for convenience,
	we get
	\begin{align*}
		Q_1 [h ; h] &= Q_1 [h ; h] + q[\xi; \xi] d \mathcal{J}_{W}(u)[
		u] \\
		&= - 2 q[\xi; \xi] \Bigl( \|w\|^2_{H^1} + \cfstr \int_{\R^3} V
		\rho_{v_+} \, dy \\
		&\qquad + \cfstr\iint_{\R^3 \times \R^3} \frac{\rho_{\psi} (y)
		\rho_{v_+}(z) - J_{\psi} (y) \cdot J_{v_+}(z) }{|y-z|} \, dy
		dz \Bigr)\\
		&\qquad - 2 q[\xi; \xi] \Bigl( \|u \|^2_{H^1} - \cfstr
		\int_{\R^3} V \rho_{\eta} \, dy \\
		&\qquad - \cfstr \iint_{\R^3 \times \R^3} \frac{\rho_{\psi}
		(y) \rho_{\eta} (z) - J_{\psi} (y) \cdot J_{\eta}(z) }{|y-z|}
		\, dy dz \Bigr) \\
		&\qquad - 2 q[\xi; \xi] \frac{ |\eta|^2_{L^2}} {1-
		|\eta|^2_{L^2}} \Gamma_{\psi}(\eta) \\
		&\leq - 2 q[\xi; \xi] (1 - Z \cfstr \gamma_T) \|w\|^2_{H^1} - 2
		q[\xi; \xi] (1 - 2 \cfstr \gamma_{T}) \|u \|^2_{H^1}
	\end{align*}  
	Now let estimate $Q_2 [h ; h] $, setting $\chi = \fw^{-1} \left(
	\begin{smallmatrix} da(u) [\xi] w_{tr} \\ h_{tr} 
	\end{smallmatrix}
	\right) = da(u) [\xi] v_{+} + \xi$.  We first note that by
	H\"older inequality implies
	\begin{align*}
		\iint_{\R^3 \times \R^3 } \frac{\RE (\psi, \chi )(y) \RE(\psi,
		\chi )(z) } {|y-z | } \, dy \, dz \leq \iint_{\R^3 \times \R^3
		} \frac{\rho_{\psi}(y) \rho_{\chi} (z) } {|y-z | } \, dy \, dz
	\end{align*}
	and by Remark \ref{rem:pos_current_term} follows
	\begin{equation*}
		\iint_{\R^3 \times \R^3 } \frac{\RE (\psi, \dalfa \chi)(y)
		\cdot \RE (\psi, \dalfa \chi )(z) } {|y-z | } \, dy \, dz \geq
		0
	\end{equation*}
	hence we have
	\begin{align*}
		Q_2 [h ; h] &\leq 2 p[\xi; \xi] \| a(u) w \|^2_{H^1} - 2 \|h
		\|^2_{H^1} + 2 \cfstr \int_{\R^3} V \rho_{\chi} (y) \, dy \\
		&\qquad + 2 \cfstr \iint_{\R^3 \times \R^3 } \frac{
		\rho_{\psi} (y) \rho_{\chi} (z) - J_{\psi}(y) \cdot J_{\chi}
		(z)} {|y-z | } \, dy \, dz\\
		&\qquad + 4 \cfstr \iint_{\R^3 \times \R^3 } \frac{
		\rho_{\psi} (y) \rho_{\chi} (z) } {|y-z | } \, dy \, dz \\
		& \leq 2 a(u)^2 p[\xi; \xi] \Bigl(\| w \|^2_{H^1} + \cfstr
		\int_{\R^3} V \rho_{v_{+}} \\
		&\qquad + \cfstr \iint_{\R^3 \times \R^3} \frac{\rho_{\psi}
		(y) \rho_{v_+}(z) - J_{\psi} (y) \cdot J_{v_+} (z) }{|y-z|}
		\Bigr)\\
		&\qquad - 2 \|h \|^2_{H^1} + 2 \cfstr \int_{\R^3} V \rho_{\xi}
		+ 2 \cfstr \iint_{\R^3 \times \R^3} \frac{\rho_{\psi} (y)
		\rho_{\xi}(z) - J_{\psi} (y) \cdot J_{\xi}(z) }{|y-z|} \\
		&\qquad - 4 r[\xi] \Gamma_{\psi}(\xi) + 4 \cfstr \iint_{\R^3
		\times \R^3 } \frac{ \rho_{\psi} (y) \rho_{\chi} (z) } {|y-z |
		} \, dy \, dz
	\end{align*}  
	Again it is convenient to add the following zero terms,
	\begin{align*}	      
		0 &= 2 r[\xi] d \mathcal{J}(u)[ h] + 2 p[\xi; \xi] d
		\mathcal{J} (u) [u ] \\
		&= - 4 p[\xi; \xi] \left( \| w \|^2_{H^1} + \cfstr \int_{\R^3}
		V \rho_{v_+} + \cfstr\iint_{\R^3 \times \R^3}
		\frac{\rho_{\psi} (y) \rho_{v_+} (z) - J_{\psi} (y) \cdot
		J_{v_+}(z) }{|y-z|} \right)\\
		& - 4 p[\xi; \xi] \left( \|u \|^2_{H^1} - \cfstr \int_{\R^3} V
		\rho_{\eta} \, dy - \cfstr \iint_{\R^3 \times \R^3}
		\frac{\rho_{\psi} (y) \rho_{\eta}(z) - J_{\psi} (y) \cdot
		J_{\eta}(z) }{|y-z|} \right) \\
		& - 4 r[\xi] \RE \langle u | h \rangle_{H^1} + 4 r[\xi] \cfstr
		\int_{\R^3} V \RE (\eta, \xi) \, dy \\
		& + 4 r[\xi] \cfstr \iint_{\R^3 \times \R^3} \frac{\rho_{\psi} (y)
		\RE(\eta, \xi) (z) - J_{\psi} (y) \cdot \RE (\eta, \dalfa
		\xi)(z) }{|y-z|} \, dy dz \\
		& + 4 r[\xi] \Gamma_{\psi}(\xi) - 4 p[\xi; \xi] \frac{
		|\eta|^2_{L^2}} {1- |\eta|^2_{L^2}} \Gamma_{\psi}(\eta).
	\end{align*}
	In view of \eqref{eq:stima_grad} $\Gamma_{\psi}(\eta) \geq 0 $ and
	by Lemma \ref{lem:key_estimate} we have
	\begin{align*}
		Q_2 [h ; h] &= Q_2 [h ; h] +2 r[\xi] d \mathcal{J}_{W}(u)[ h]
		+ 2 p[\xi; \xi] d \mathcal{J}_{W} (u) [u ] \\
		&\leq - 2 p[\xi; \xi] \left( \| w \|^2_{H^1} + \cfstr
		\int_{\R^3} V \rho_{v_+} + \cfstr \iint_{\R^3 \times \R^3}
		\frac{\rho_{\psi} (y) \rho_{v_+} (z) - J_{\psi} (y) \cdot
		J_{v_+}(z) }{|y-z|} \right)\\
		&\quad - 2 \|h \|^2_{H^1} + 2 \cfstr \int_{\R^3} V \rho_{\xi}
		\, dy + 2 \cfstr \iint_{\R^3 \times \R^3} \frac{\rho_{\psi}
		(y) \rho_{\xi}(z) - J_{\psi} (y) \cdot J_{\xi}(z) }{|y-z|} \,
		dy dz \\
		&\quad - 4 p[\xi; \xi] \left( \|u \|^2_{H^1} - \cfstr
		\int_{\R^3} V \rho_{\eta} \, dy - \cfstr \iint_{\R^3 \times
		\R^3} \frac{\rho_{\psi} (y) \rho_{\eta}(z) - J_{\psi} (y)
		\cdot J_{\eta} (z) }{|y-z|} \, dy dz \right) \\
		&\quad + 4 r[\xi] ( - \RE \langle u |h \rangle_{H^1} +
		\cfstr \int_{\R^3} V \RE (\eta, \xi) )\\
		&\quad + 4 r[\xi] \cfstr \iint_{\R^3 \times \R^3}
		\frac{\rho_{\psi} (y) \RE( \eta, \xi) (z) - J_{\psi} (y) \cdot
		\RE( \eta, \dalfa \xi)(z) }{|y-z|} \\
		&\quad + 4 \cfstr \iint_{\R^3 \times \R^3 } \frac{ \rho_{\psi}
		(y) \rho_{\chi} (z) } {|y-z | } \, dy \, dz \\
		&\leq- 2 p[\xi; \xi] (1 - Z \cfstr \gamma_T) \| w \|^2_{H^1} - (1 - 2
		\cfstr \gamma_{T}) \|h \|^2_{H^1} - (1- 2 \cfstr \gamma_{T})
		\|h + 2 r[\xi] u \|^2_{H^1} \\
		&\qquad + 4 \cfstr \gamma_{K} | da(u) [\xi] v_{+} + \xi
		|^2_{H^{1/2}}
	\end{align*}
	Therefore, since $q[\xi; \xi] \geq 2p[\xi; \xi]$, we have
	\begin{align*}
		Q_1 [h ; h] + Q_2 [h ; h] &\leq - 6 (1- Z \cfstr \gamma_T) p[\xi;
		\xi] \|w\|^2_{H^1} \\
		&\qquad - (1 - 2 \cfstr \gamma_{T}) \|h \|^2_{H^{1}} + 4
		\cfstr \gamma_{K} ( p[\xi; \xi] \|a(u) w\|^2_{H^1} + \| h
		\|^2_{H^1} ) \\
		& \leq - 6 (1- Z \cfstr \gamma_T - \frac{2}{3} \cfstr \gamma_{K} )
		p[\xi; \xi] \| w\|^2_{H^1} - (1 - 2 \cfstr \gamma_{T} - 4
		\cfstr \gamma_{K}) \|h \|^2_{H^1} \\
		& \leq - (1 - 8 \cfstr \gamma_{T} ) \|h \|^2_{H^1}
	\end{align*}
	where we have used that $\gamma_{K} < \frac{3}{2} \gamma_{T}$,
	$Z \cfstr \gamma_T + \cfstr \gamma_{T} = \cfstr (Z + 1) \gamma_{T} \leq
	1$ (since $Z \leq 123$).
\end{proof} 
In view of the above results we may conclude
\begin{prop}
	\label{prop:sup_achieved}
	For any $w \in H^{1}(\R^4_{+}, \mathbb{C}^2)$ with $|w_{tr}|_{L^2}
	= 1$ there exists unique $ \phi_2 = \phi_2(w) $, a strict global
	maximum of $\mathcal{J}_{W}$, namely 
	\begin{equation*}
		\mathcal{J}_{W}(\phi_{2}) = \sup_{u \in B_{1}}
		\mathcal{J}_{W}(u) = \sup_{ \phi \in \mathcal{X}_{W} }
		\mathcal{I}(\phi) = \lambda_{W}.
	\end{equation*}
	Moreover 
	\begin{itemize}
		\item $d\mathcal{J}_{W}( \phi_2(w)) = 0$;
		
		\item there exists $\delta >0$ such that
		\begin{equation*}
			\mathcal{J}_{W}( \phi_2(w))[h ; h] \leq - \delta \|h\|^2
			\qquad \forall h \in H^1
		\end{equation*}
		
		\item the map $w \to \phi_2(w)$, is smooth and
		\begin{equation*}
			d\phi_2 (w)[dP(w)[ \cdot ]] = - (d_{u} F(w,
			\phi_2(w)))^{-1} [d_{w} F(w, \phi_2(w))[\cdot ]] .
		\end{equation*}
		where $P(w) = \frac{w}{|w_{tr}|_{L^2}}$ and 
		\begin{equation*}
			F(w, u)[h] = d \mathcal{I}\left(\left( \begin{smallmatrix}
			a(u) P( w )\\ u \end{smallmatrix} \right) \right) \left[
			\left(
			\begin{smallmatrix} 
				da(u)[h ] P(w) \\ h
			\end{smallmatrix} 
			\right) \right] \qquad \forall h \in H^1.
		\end{equation*}
	\end{itemize}
\end{prop}

\begin{proof}
	It is clear that the equality $\sup_{u \in B_{1}}
	\mathcal{J}_{W}(u) = \sup_{ \phi \in \mathcal{X}_{W}}
	\mathcal{I}(\phi)$ holds.  The existence of a maximizer for
	$\mathcal{J}_{W}$ then follows from lemma
	\ref{lem:apriori_estimates}, which shows that the supremum is
	strictly positive, Ekeland's variational principle, which implies
	that we can find a maximizing sequence $u_{n}$ which is also a
	Palais-Smale sequence, and Proposition \ref{prop:PalaisSmale2},
	which shows that $u_{n} \to u$ with $\abs{u_{\text{tr}}}_{L^{2}} <
	\frac{1}{2}$.
	
	Suppose that we have another maximizer $\tilde{u}$.  By
	Proposition \ref{prop:PalaisSmale2} we deduce that
	$\abs{\tilde{u}_{\text{tr}}}_{L^{2}} < \frac{1}{2}$.
	
	To reach a contradiction, we consider the set
	\begin{equation*}
		\mathcal{G} = \settc{g \colon [-1,1] \to
		H^{1}(\mathbb{R}^{4}_{+}, \mathbb{C}^{2})}{g(-1) = u, \ g(1) =
		\tilde{u}, \ \abs{(g(t))_{\text{tr}}}_{L^{2}} \leq
		\frac{1}{2}}
	\end{equation*}
	and the min-max level
	\begin{equation*}
		c = \sup_{g \in \mathcal{G}} \min_{t \in [-1,1]}
		\mathcal{J}_{W}(g(t))
	\end{equation*}
	The functional $\mathcal{J}_{W}$ satisfies the Palais-Smale
	condition, see proposition \ref{prop:PalaisSmale2}, and the set
	$B_{1/2}{} = \settc{u \in H^{1}(\R^4_{+},
	\mathbb{C}^2)}{\abs{u_{\text{tr}}}_{L^{2}} < 1/2}$ is invariant
	for the gradient flow generated by $d\mathcal{J}_{W}$ since by
	\eqref{eq:stima_gradPalla} $d\mathcal{J}_{W}(u)[u] < 0$ on
	$\partial B_{1/2}$.  Then we can deduce that $c$ is a Mountain
	pass critical level, and that there is a Mountain pass critical
	point $\tilde{v}$ in $B_{1/2}$, i.e.~such that
	$\abs{\tilde{v}_{\text{tr}}}_{L^{2}} < \frac{1}{2}$, a
	contradiction with Proposition \ref{prop:concavity}, since a
	Mountain pass critical point cannot be a strict local maximum.

	Finally to prove that the map $w \to \phi_2(w)$ is smooth we use
	the implicit function theorem.  Indeed let consider any open
	subset $U \subset H^1 \setminus \{w_{tr} = 0 \}$ and the smooth
	map $F : U \times H^1 \to H^{-1}$ defined by
	\begin{equation*}
		F(w, u)[h] = d \mathcal{I}\left(\left( \begin{smallmatrix}
		a(u) P( w )\\ u \end{smallmatrix} \right) \right) \left[
		\left( \begin{smallmatrix} da(u)[h] P(w) \\ h
		\end{smallmatrix} \right) \right] \qquad \forall h \in H^1.
	\end{equation*}

	Now fix $w_0 \in U$ with $|w_0|_{L^2} = 1$ and let $u_0 =
	\phi_2(w_0)$ and $W_0 = \spann \{ w_0 \}$, we have
	\begin{equation*}
		F(w_0, u_0) = d\mathcal{J}_{W_0}(u_0) = 0
	\end{equation*}
	and the operator $d_u F(w_0, u_0) : H^1 \to H^{-1}$ given by
	\begin{align*}
		(d_u F (w_0, u_0)[h]) [k] = & d^2 \mathcal{I} \left(\left(
		\begin{smallmatrix} a (u_0) w_0\\ u_0\end{smallmatrix} \right)
		\right) \left[ \left( \begin{smallmatrix} da (u_0) [h] w_0\\ h
		\end{smallmatrix} \right) ; \left( \begin{smallmatrix} da (
		u_0) [k] w_0\\ k \end{smallmatrix} \right) \right] \\
		&+ d \mathcal{I} \left( \left( \begin{smallmatrix} a ( u_0)
		w_0 \\ u_0 \end{smallmatrix} \right) \right) \left[
		\left(\begin{smallmatrix} d^2 a ( u_0) [h; k] w_0 \\
		0\end{smallmatrix} \right)\right] \qquad \forall h, k \in H^1
	\end{align*}
	is invertible.  Indeed, we simply apply the Riesz theorem on
	Hilbert spaces (or equivalently Lax-Milgram theorem) to the
	symmetric, bilinear and bi-continuous, quadratic form $Q : H^1
	\times H^1 \to \R$ given by
	\begin{equation*}
		Q[h;k] = - (d_u F (w_0, u_0)[h]) [k].
	\end{equation*}
	In view of Proposition \ref {prop:concavity}
	\begin{equation*}
		Q[h;h] = - d^2 \mathcal{J}_{W_0}( u_0)[h; h] \geq \delta
		\|h\|^2 \qquad \forall h \in H^1
	\end{equation*}
	for some $\delta >0$, namely $Q$ is definite positive (coercive)
	and the theorem apply, hence for any $f \in H^{-1}$ there exists
	unique $h \in H^1$ such that $Q[h;k] = f[k]$ for any $k \in H^1$,
	namely $d_u F(w_0, u_0)[h] = - f$.
 
	Therefore we can apply the implicit function theorem to conclude
	that there exists a neighborhood $U_0 \subset X_{+} \setminus
	\{w_{tr} = 0 \}$ of $w_0$ and a smooth map $u : U_0 \to H^1$ such
	that $F(w, u(w)) = 0 $ for all $w \in U_0$.
	
	Since we already know that for any $w \in X_{+} \setminus \{
	w_{tr} = 0 \} $ there exist $ \phi_2(P(w))$, the unique strict
	global maximum of $\mathcal{J}_{W}$, such that $F(w, \phi_2(P(w)))
	= 0$, we may conclude that $u(w) \equiv \phi_2(P(w))$ for any $w
	\in U_0$.

	Moreover, we have that for $w \in U_0$, $du(w): H^1 \to H^1$ is
	given by
	\begin{equation*}
		du(w)[h] = - (d_u F(w, u(w)))^{-1} [d_{w} F(w, u(w))[h]]
		\qquad \forall h \in H^1.
	\end{equation*}
\end{proof}    
    
\begin{cor}
	\label{cor:sup_achieved_I}
	For any $w \in H^{1}(\R^4_{+}, \mathbb{C}^2)$ with $|w_{tr}|_{L^2}
	= 1$, let $\bar{\phi}(w) = \left( \begin{smallmatrix} a(\phi_2(w)
	) w \\ \phi_2(w) \end{smallmatrix} \right)$.
	
	Then $\bar{\phi}(w) \in \mathcal{X}_{W} $ is the unique (up to
	phase) maximizer of $\mathcal{I}$ in $\mathcal{X}_{W} $, namely
	\begin{equation}
		\label{eq:p11}
		\mathcal{I}(\bar{\phi}(w)) = \sup_{ \phi \in \mathcal{X}_{W} }
		\mathcal{I}(\phi) = \lambda_{_{W}} > 0 .
	\end{equation}
	and 
	\begin{equation*}
		d \mathcal{I}(\bar{\phi}(w)) [h] = 2 \mult(\bar{\phi}(w)) \RE
		\langle(\bar{\phi}(w))_{tr} \mid h_{tr} \rangle_{L^2} \qquad
		\forall h \in W \oplus X_{-}
	\end{equation*}
	where 
	\begin{equation*}
		\mult(\bar{\phi}(w)) = \lambda_{W} + \frac{\cfstr }{2}
		\iint_{\R^3 \times \R^3}
		\frac{\rho_{\psi_w}(y)\rho_{\psi_w}(z) - J_{\psi_w}(y) \cdot
		J_{\psi_w}(z)}{|y-z|} \, dy dz .
	\end{equation*}
	and $\psi_{w} = \fw^{-1} (\bar{\phi}(w))_{\text{tr}}$.  Moreover,
	the following estimates holds
	\begin{itemize}
 
		\item[(i)] $ | \phi_{1}(w)_{\text{tr}} |^2_{L^2} >
		\frac{1}{2}$;
   
		\item [(ii)] $ \|\bar{\phi}(w) \|^2_{H^1} \leq \frac{1 +
		\cfstr\gamma_{T}Z}{(1 - \cfstr\gamma_{T})(1 -
		\cfstr\gamma_{T}Z)} \lambda_{W}$.
 
	\end{itemize}
	
\end{cor}
    
\begin{proof}
	We only have to show that item (ii) holds. 
	
	If $\phi(w)$ is the maximizer for $\mathcal{I}$ in $
	\mathcal{X}_{W}$ we have as in the proof of Lemma
	\ref{lem:PalaisSmale},
	\begin{equation*} 
		\lambda_{W} \leq (1 + \cfstr\gamma_{T}) \|\phi_{1}(w)
		\|^2_{H^1} - (1 - \cfstr\gamma_{T}) \|\phi_{2}(w) \|^2_{H^1}
	\end{equation*}
	Moreover we have
	\begin{equation*}
		\lambda_{W} \geq \mathcal{I}( \left( \begin{smallmatrix}
		\phi_{1}(w) \\ 0\end{smallmatrix} \right) ) \geq \|\phi_{1}(w)
		\|^2_{H^1} + \cfstr \int_{\R^3} V \rho_{\psi_{+,w}} \, dy \geq
		(1 - \cfstr\gamma_{T}Z) \|\phi_{1}(w) \|^2_{H^1}
	\end{equation*}
	where $\psi_{+,w} = \fw^{-1} \left( \begin{smallmatrix}
	\phi_{1}(w) \\ 0\end{smallmatrix} \right)$. Hence we may conclude
	that
	\begin{align*}
		&\|\phi_{2}(w) \|^2_{H^1} \leq \frac{\cfstr\gamma_{T}}{1 -
		\cfstr\gamma_{T}} (1 + Z) \|\phi_{1}(w) \|^2_{H^1}\\
		&\|\phi_{1}(w) \|^2_{H^1} \leq \frac{ \lambda_{W}}{ 1 -
		\cfstr\gamma_{T}Z }
	\end{align*}
	and also
	\begin{equation*} 
		\|\phi_{1}(w) \|^2_{H^1}+ \|\phi_{2}(w) \|^2_{H^1} \leq \frac{1
		+ \cfstr\gamma_{T}Z}{1 - \cfstr\gamma_{T}}
		\norm{\phi_{1}(w)}^{2} \leq \frac{1 + \cfstr\gamma_{T}Z}{(1 -
		\cfstr\gamma_{T})(1 - \cfstr\gamma_{T}Z)} \lambda_{W}.
	\end{equation*}
\end{proof}

\section {Proof of Theorem \ref{thm:main}}
 
In view of the results of Proposition \ref{prop:sup_achieved} it is
convenient to introduce the smooth functional $ \mathcal{F}:
H^1(\R^4_{+}; \C^2) \setminus \{w_{tr} \equiv 0 \} \to \R$
\begin{equation*}
	\mathcal{F}(w) = \mathcal{I} (\phi(w) )
\end{equation*} 
where $\phi(w) = \left( \begin{smallmatrix} a( \phi_2(P(w)) ) P(w) \\
\phi_2(P(w)) \end{smallmatrix} \right) $ and $P(w) =
\frac{w}{|w_{tr}|_{L^2}}$.  Now in view of Proposition
\ref{prop:sup_achieved} we may conclude that
\begin{equation*}
	\Lambda_1 = \inf_{ \substack{ W \subset X_{+} \\ \dim W = 1}} \,
	\sup_{ \phi \in \mathcal{X}_{W} } \mathcal{I}(\phi) =
	\inf_{ \substack { w \in H^1(\R^4_{+}; \C^2) \\
	|w_{tr}|_{L^2} = 1} }
	\mathcal{I}(\bar{\phi}(w)) = \inf_{w \in H^1(\R^4_{+}; \C^2)
	\setminus \{w_{tr} = 0 \}} \mathcal{F}(w)
 \end{equation*}

Let us introduce the constraint manifold $\mathcal{W} \subset
H^1(\R^4_{+}; \C^2)$
\begin{equation*}
	\mathcal{W} = \{ w \in H^1(\R^4_{+}; \C^2) \, : \, G(w) :=
	|w_{tr}|^2_{L^2} - 1 = 0 \}
\end{equation*}
and its tangent space 
\begin{equation*}
	T_{w} \mathcal{W} = \{ h \in H^1(\R^4_{+}; \C^2) \, : \, dG(w)[h]
	= 2 \RE \langle w_{tr}| h_{tr} \rangle_{L^2} = 0 \}
\end{equation*}

Let us compute $d \mathcal{F}(w)[h] = d \mathcal{I}
(\phi(w))[d\phi(w)[h]]$. For $w \in \mathcal{W}$ and $h \in
H^1(\R^4_{+}; \C^2)$, we have
\begin{align*}
	d \phi(w)[h] = \left( \begin{smallmatrix}
	da(\phi_2(w))[d\phi_2(w)[dP(w)[h]] w \\
	d\phi_{2}(w)[dP(w)[h]] \end{smallmatrix} \right) + \left(
	\begin{smallmatrix} a(\phi_2(w)) dP(w)[h] \\ 0 \end{smallmatrix}
	\right).
\end{align*}
Since 
\begin{equation*}
	0= d \mathcal{J}_W (\phi_2(w)) [k] = d \mathcal{I} (\phi(w))
	\left[\left( \begin{smallmatrix} da(\phi_2(w)) [k] w \\ k
	\end{smallmatrix} \right) \right] \qquad \forall k \in H^1
\end{equation*} 
we have
\begin{align*}
	d \mathcal{F}(w) [h] &= d \mathcal{J}_W (\phi_2(w))
	[d\phi_2(w)[dP(w)[h]] + d \mathcal{I} (\phi(w)) \left[\left(
	\begin{smallmatrix} a(\phi_2(w)) dP(w)[h] \\ 0 \end{smallmatrix}
	\right) \right] \\
	&= a(\phi_2(w) ) d \mathcal{I} (\phi(w)) \left[\left(
	\begin{smallmatrix} dP(w)[h] \\ 0 \end{smallmatrix} \right)
	\right] .
\end{align*}
where $dP(w)[h] = h - w \RE \langle w_{tr} \mid h_{tr} \rangle_{L^2}$
if $w \in \mathcal{W}$.

Since $ d \mathcal{F}(w) [w] = 0$ for any $w \in \mathcal{W}$, it is
easy to see that $\mathcal{W}$ is indeed a natural constraint for
$\mathcal{F}$.  Hence in particular by Ekeland's variational
principle, there exists a Palais-Smale, minimizing sequence $\{ w_n \}
\in \mathcal{W}$, namely $\mathcal{F}(w_n) \to \Lambda_{1}$ and $\|
d\mathcal{F}(w_n) \| \to 0$.

Now setting $\phi_n = \phi(w_n) = \left( \begin{smallmatrix}
\phi_{1,n} \\ \phi_{2,n} \end{smallmatrix} \right) $, with $\phi_{1,n}
= a_n w_n$, $a_n = a(\phi_2(w_n) )$, $\phi_{2,n} = \phi_2(w_n) $ and
$\mult_n = \mult(\phi_n) $ and defining the linear continuous
functional $\mathcal{T}_n \colon H^{1}(\R^{4}_+ ; \C^{2}) \to 
\R$
\begin{equation}
	\label{eq:weak_eq_n}
	\mathcal{T}_n [h] = d \mathcal{I}(\phi_n) [ \left(
	\begin{smallmatrix} h \\ 0 \end{smallmatrix} \right) ] - 2 \mult_n
	\RE \langle a_n ( w_n)_{tr} \mid h_{tr} \rangle_{L^2}
\end{equation}
in view of Corollary \ref{cor:sup_achieved_I} we have that
$\mathcal{T}_n [h] = 0$ for any $h \in \spann \{ w_n \}$ and $n \in
\mathbb{N}$.

On the other hand for any $h \in H^{1}(\R^{4}_+ ; \C^{2})$
\begin{equation*}
	d \mathcal{F}(w_n) [h] = a_n d \mathcal{I} (\phi_n) \left[\left(
	\begin{smallmatrix} dP(w_n)[h] \\ 0 \end{smallmatrix} \right)
	\right] = a_n \mathcal{T}_n [dP(w_n)[h]] = a_n \mathcal{T}_n [h]
\end{equation*}
and since $\| d \mathcal{F}(w_n) \| \to 0$ and $a_n > \frac{1}{2}$ we
may conclude that $\mathcal{T}_{n} \to 0 $ strongly.


Since the sequence $\{ \phi_{n} \}$ is bounded in $H^1(\R^{4}_+
;\C^{4} )$ (it follows from Corollary \ref{cor:sup_achieved_I} since
$\lambda_{W_{n}} \to \Lambda_{1}$) we get that, up to a subsequence,
$\phi_n \rightharpoonup \phi $ weakly in $H^1$ and $ \mult_n \to
\mult$, and hence,
\begin{equation*} 
	d \mathcal{I}(\phi) [h] = 2 \mult \RE \langle \phi_{tr} ,h_{tr}
	\rangle_{L^2} \qquad \forall h \in H^1(\R^{4}_+ ; \C^{4} ).
\end{equation*}
(since $d \mathcal{I}(\phi_{n})[h] \to d \mathcal{I}(\phi)[h]$ if
$\phi_{n} \rightharpoonup \phi$, see \eqref{eq:convergenceto0}).
 
To conclude the proof of Theorem \ref{thm:main} we need to show that
$|\phi_{tr}|_{L^2} = 1$, that is a strong convergence in $L^2$ of
$(\phi_n)_{tr}$, in fact we will prove strong convergence of $\phi_n$
in $H^1$.


First note that we can assume that
\begin{equation}
	\label{eq:Hessian_condition}
	\liminf_{n \to + \infty} d^2 \mathcal{F}(w_n)[h ; h] \geq 0 \qquad
	\forall h \in H^1(\R^{4}_+ ; \C^{2} ).
\end{equation}
This is an adaptation of of theorem 2.6 in Borwein and Preiss 
\cite{zbMATH04028246} with $p = 2$, $\epsilon = \frac{1}{n}$ and 
$\lambda = 1$ (see also \cite{zbMATH06266957})
which states that one can find a minimizing sequence such that
\begin{align*}
	&\mathcal{F}(w_{n}) \leq \inf_{w \in \mathcal{W}} \mathcal{F}(w) + 
	\frac{1}{n}, \\
	&\mathcal{F}(w_{n}) + \frac{1}{{n}} \Delta(w_{n}) \leq
	\mathcal{F}(w) + \frac{1}{{n}} \Delta(w) \qquad \text{for all } w
	\in \mathcal{W}
\end{align*}
where $\Delta(y) = \sum_{k=1}^{\infty} \beta_{k} \norm{y -
y_{k}}_{H^{1}}^{2}$ for a (convergent) sequence of points $y_{k}$ and
reals $\beta_{k} \geq 0$ such that $\sum_{k=1}^{\infty} \beta_{k} =
1$.  The above relation shows that $w = w_{n}$ is a minimizer for
$G_{n}(w) = \mathcal{F}(w) + \frac{1}{{n}} \Delta(w)$ and hence
\begin{equation*}
	0 \leq d^{2}G_{n}(w_{n})[h,h] = d^{2}\mathcal{F}(w_{n})[h,h] +
	\frac{1}{{n}} \langle h \mid h \rangle_{H^{1}}.
\end{equation*}

Now with the additional information \eqref{eq:Hessian_condition} on
the second variations, we prove the following bound on the Lagrange
multiplier $\mult$, that it will be a key point to prove strong
convergence of the minimizing sequence.  We have

\begin{lem}
	$\mult < 1$
\end{lem}

\begin{proof}
	 Since $w_{n}$ is bounded we can assume that $w_{n}
	 \rightharpoonup w$ in $H^{1}$.  Take $h \in H^1(\R^4_{+}; \C^2)$
	 such that $h_{tr} \in H^1(\R^3; \C^2)$ and $\RE \langle w_{tr} |
	 h_{tr} \rangle_{L^2} =0 $.

	We set $h_{1,n} = dP(w_n)[h]$ and $h_{2,n}=
	d\phi_2(w_n)[dP(w_n)[h]]$, then we have
	\begin{align*}
		d^2 \mathcal{F}(w_n) &[h ; h] = d^2 \mathcal{I} (\phi_n)
		\left[ \left( \begin{smallmatrix} da(\phi_{2,n})[h_{2,n} ] w_n
		\\ h_{2,n} \end{smallmatrix} \right); \left(
		\begin{smallmatrix} a_n h_{1,n} \\ 0 \end{smallmatrix}
		\right)\right] \\
		& + d^2 \mathcal{I} (\phi_n) \left[ \left( \begin{smallmatrix}
		a_n h_{1,n} \\ 0 \end{smallmatrix} \right); \left(
		\begin{smallmatrix} a_n h_{1,n} \\ 0 \end{smallmatrix}
		\right)\right] + d \mathcal{I} (\phi_n) \left[\left(
		\begin{smallmatrix} da(\phi_{2,n})[h_{2,n} ] h_{1,n} \\ 0
		\end{smallmatrix} \right) \right] \\
		& + d \mathcal{I} (\phi_n) \left[\left( \begin{smallmatrix}
		a_n d^2P(w_n)[h;h] \\ 0 \end{smallmatrix} \right) \right] =
		(I) + (II) + (III).
	\end{align*}
	where 
	\begin{align*}
		&(I) = d^2 \mathcal{I} (\phi_n) \left[ \left(
		\begin{smallmatrix} da(\phi_{2,n})[h_{2,n} ] w_n \\ h_{2,n}
		\end{smallmatrix} \right); \left( \begin{smallmatrix} a_n
		h_{1,n} \\ 0 \end{smallmatrix} \right)\right] + d \mathcal{I}
		(\phi_n) \left[\left( \begin{smallmatrix}
		da(\phi_{2,n})[h_{2,n} ] h_{1,n} \\ 0 \end{smallmatrix}
		\right) \right]\\
		&(II) = d^2 \mathcal{I} (\phi_n) \left[ \left(
		\begin{smallmatrix} a_n h_{1,n} \\ 0 \end{smallmatrix}
		\right); \left( \begin{smallmatrix} a_n h_{1,n} \\ 0
		\end{smallmatrix} \right)\right] \\
		& (III) = d \mathcal{I} (\phi_n) \left[\left(
		\begin{smallmatrix} a_n d^2P(w_n)[h;h] \\ 0 \end{smallmatrix}
		\right) \right] .
	\end{align*}
	In view of Proposition \ref{prop:sup_achieved} for any $w \in
	\mathcal{W}$ we have for all $h \in H^{1}$
	\begin{equation*}
		h_{2,n} = d\phi_{2}(w)[dP(w)[h] ]= - (d_u F(w, \phi_2(w))^{-1}
		[(d_{w} F(w, \phi_2(w))[h ])],
	\end{equation*} 
	where the map $F : H^1 \setminus \{w_{tr} = 0 \} \times H^1 \to
	H^{-1}$ is given by
	\begin{equation*}
		F(w, u)[k] = d \mathcal{I}\left(\left( \begin{smallmatrix}
		a(u) P( w )\\ u \end{smallmatrix} \right) \right) \left[
		\left( \begin{smallmatrix} da(u)[k] P(w) \\ k
		\end{smallmatrix} \right) \right] \qquad \forall k \in H^1,
	\end{equation*}
	let compute the operator $d_{w} F(w, u): H^1 \to H^{-1}$, for $w
	\in \mathcal{W}$ and any for $h_1, h_2 \in H^1$ we have
	\begin{align*}
		(d_{w} F(w, u)[h_1 ])[h_2] = & d^2 \mathcal{I}\left(\left(
		\begin{smallmatrix} a(u) w \\ u \end{smallmatrix} \right)
		\right) \left[ \left( \begin{smallmatrix} a(u) dP(w) [h_1] \\
		0 \end{smallmatrix} \right) ; \left( \begin{smallmatrix}
		da(u)[h_2] w \\ h_2 \end{smallmatrix} \right) \right] \\
		&+ d \mathcal{I}\left(\left( \begin{smallmatrix} a(u) w \\ u
		\end{smallmatrix} \right) \right) \left[ \left(
		\begin{smallmatrix} da(u)[h_2] dP(w)[h_1] \\ 0
		\end{smallmatrix} \right) \right].
	\end{align*}
	Hence we have
	\begin{equation*}
		- (d_u F(w_n, \phi_{2,n}) [h_{2,n} ]) [h_{2,n} ] = (d_{w}
		F(w_n, \phi_{2,n})[h])[h_{2,n} ] = (I) .
	\end{equation*}
	Recalling that 
	\begin{align*}
		(d_u F (w, u)[k]) [k] = & \, d^2 \mathcal{I} \left(\left(
		\begin{smallmatrix} a (u) w \\ u \end{smallmatrix} \right)
		\right) \left[ \left( \begin{smallmatrix} da (u) [k] w \\ k
		\end{smallmatrix} \right) ; \left( \begin{smallmatrix} da (u)
		[k] w \\ k \end{smallmatrix} \right) \right] \\
		&+ d \mathcal{I} \left( \left( \begin{smallmatrix} a ( u) w \\
		u \end{smallmatrix} \right) \right) \left[
		\left(\begin{smallmatrix} d^2 a (u) [k; k] w \\
		0\end{smallmatrix} \right)\right] = d^2 \mathcal{J}_{W}(u)[k ;
		k] \qquad \forall k \in H^1 .
	\end{align*}
	in view of Proposition \ref{prop:concavity}, we get
	\begin{equation*}
		(I) = (d_{w} F(w_n, \phi_{2,n})[h])[ h_{2,n}] = - d^2
		\mathcal{J}_{W}(\phi_{2,n})[h_{2,n} ; h_{2,n} ] \geq \delta
		\|h_{2,n}\|^2_{H^1}.
	\end{equation*}
	On the other hand, we have
	\begin{align*}
		(I) = & a_n d^2 \mathcal{I} (\phi_n) \left[ \left(
		\begin{smallmatrix} 0 \\ h_{2,n} \end{smallmatrix} \right);
		\left( \begin{smallmatrix} h_{1,n} \\ 0 \end{smallmatrix}
		\right)\right] - da( \phi_{2,n} )[h_{2,n} ] d^2 \mathcal{I}
		(\phi_n) \left[ \left( \begin{smallmatrix} 0 \\ \phi_{2,n}
		\end{smallmatrix} \right); \left( \begin{smallmatrix}
		h_{1,n} \\ 0 \end{smallmatrix} \right)\right] \\
		&+ da( \phi_{2,n} )[h_{2,n} ] \left( d^2 \mathcal{I} (\phi_n )
		\left[ \phi_n; \left( \begin{smallmatrix} h_{1,n} \\ 0
		\end{smallmatrix} \right)\right] + d \mathcal{I}(\phi_n)
		\left[ \left( \begin{smallmatrix} h_{1,n}\\ 0
		\end{smallmatrix} \right) \right] \right).
	\end{align*}
	Then, since $\langle (w_n)_{tr} | h_{1,n} \rangle_{L^2} = \langle
	(w_n)_{tr} |dP(w_n)[h] \rangle_{L^2} =0$, by Corollary
	\ref{cor:sup_achieved_I} we have
	\begin{equation*}
		d \mathcal{I} (\phi_n ) \left[ \left( \begin{smallmatrix}
		h_{1,n} \\ 0 \end{smallmatrix} \right) \right] = \mathcal{T}_n
		(h)
	\end{equation*}
	and recalling that $h_{1,n} =dP(w_n)[h] = h - w_n \RE \langle
	(w_n)_{tr} | h_{tr} \rangle_{L^2} \to h $, as $n \to + \infty$
	strongly in $H^1$, we set $\xi = \fw^{-1} \left(
	\begin{smallmatrix} h_{tr} \\ 0 \end{smallmatrix} \right) $, by
	H\"older's and Hardy's inequalities, we get
	\begin{align*}
		d^2 \mathcal{I} (\phi_n ) & \left[ \phi_n; \left(
		\begin{smallmatrix} h_{1,n} \\ 0 \end{smallmatrix} \right)
		\right] = d \mathcal{I} (\phi_n ) \left[ \left(
		\begin{smallmatrix} h_{1,n} \\ 0 \end{smallmatrix} \right)
		\right] \\
		&+ 4 \cfstr \iint_{\R^3 \times \R^3 } \frac{\rho_{\psi_n}(y)
		\RE( \psi_n, \xi)(z) - J_{\psi_n}(y) \cdot \RE (\psi_n,
		\dalfa\xi)(z) } {|y-z | } \, dy \, dz + o_n(1) \\
		&\leq \mathcal{T}_n (h) + 8 \cfstr \iint_{\R^3 \times \R^3 }
		\frac{\rho_{\psi_n}(y) |\psi_n|(z) | \xi|(z) } {|y-z | } \, dy
		\, dz + o_n(1) \\
		&\leq \mathcal{T}_n (h) + C \int_{\R^3 } \rho_{\psi_n}(y)
		\left( \int_{\R^3 } \frac{ |\xi|^2(z) } {|y-z |^2 } \, dz
		\right)^{1/2} dy + o_n(1) \\
		& \leq \mathcal{T}_n (h) + C | \nabla \xi |_{L^2} + o_n(1).
	\end{align*}
	and analogously, by H\"older and Hardy's inequalities, we
	have
	\begin{equation*}
		d^2 \mathcal{I} (\phi_n) \left[ \left( \begin{smallmatrix} 0
		\\ h_{2,n} \end{smallmatrix} \right); \left(
		\begin{smallmatrix} h_{1,n} \\ 0 \end{smallmatrix}
		\right)\right] \leq C \| h_{2,n} \|_{H^1} ( | \nabla \xi
		|_{L^2} + o_n(1))
	\end{equation*}
	and 
	\begin{equation*}
		\left| d^2 \mathcal{I} (\phi_n) \left[ \left(
		\begin{smallmatrix} 0 \\ \phi_{2,n} \end{smallmatrix}
		\right); \left( \begin{smallmatrix} dP(w_n)[h] \\ 0
		\end{smallmatrix} \right)\right] \right| \leq C ( | \nabla \xi
		|_{L^2} + o_n(1) )
	\end{equation*}
	for some constant $C >0$ that may change from line to line.
 
	Hence, since $da( \phi_{2,n} )[h_{2,n} ] \leq |(h_{2,n})_{tr}
	|_{L^2} \leq \|h_{2,n}\|_{H^1} $ and $|a_n| \leq 1$, we get
	\begin{equation*}
		\delta \|h_{2,n}\|^2_{H^1} \leq (I) \leq C \|h_{2,n}\|_{H^1} (
		| \nabla \xi |_{L^2} + o_n(1))
	\end{equation*}
	namely
	\begin{equation*}
		\|h_{2,n}\|_{H^1} \leq C ( | \nabla \xi |_{L^2} + o_n(1))
	\end{equation*}
	and we may conclude that
	\begin{equation*}
		(I) \leq C | \nabla \xi |^2_{L^2} + o_n(1).
	\end{equation*}
	Now, by Remark \ref{rem:pos_current_term} and H\"older inequality,
	we have
	\begin{multline*}
		(II) \leq 2 a_n^2 \| h\|^2_{H^1} + 2 a_n^2 \cfstr
		\int_{\R^{3}} V \rho_{\xi} \, dy \\
		+ 8 a_n^2 \cfstr \iint_{\R^3 \times \R^3 }
		\frac{\rho_{\psi_n}(y) \rho_{\xi}(z) } {|y-z | } \, dy \, dz +
		o_n(1)
	\end{multline*}
	Moreover, recalling that 
	\begin{align*}
		d^2P(w_n)[h;h] = & 3 |\RE \langle ( w_n)_{tr} | h_{tr}
		\rangle_{L^2}|^2 w_n - 2 \RE \langle ( w_n)_{tr} | h_{tr}
		\rangle_{L^2} h - |h_{tr}|^2_{L^2} w_n
	\end{align*}
	we have by Corollary \ref{cor:sup_achieved_I}
	\begin{align*}
		(III) &= \mult_n 2 a_n^2 \RE\langle ( w_n)_{tr} |
		(d^2P(w_n)[h;h] )_{tr} \rangle_{L^2} + a_n
		\mathcal{T}_n(d^2P(w_n)[h;h])\\
		&= \mult_n 2 a_n^2 ( |\RE \langle ( w_n)_{tr} | h_{tr}
		\rangle_{L^2}|^2 - |h_{tr}|^2_{L^2} ) + o_n(1)\\
		& = - 2 a_n^2 \mult_n |h_{tr}|^2_{L^2} + o_n(1).
	\end{align*}

	Collecting the estimates above we get
	\begin{align*}
		d^2 \mathcal{F}(w_n) [h ; h] \leq & 2 a_n^2 \left( \| h
		\|^2_{H^1} - \mult_n |h_{tr}|^2_{L^2} \right)+ 2 a_n^2 \cfstr
		\int_{\R^{3}} V \rho_{\xi} \, dy \\
		&+ 8 a_n^2 \cfstr \iint_{\R^3 \times \R^3 }
		\frac{\rho_{\psi_n}(y) \rho_{\xi}(z) } {|y-z | } \, dy \, dz +
		C | \nabla \xi |^2_{L^2} + o_n(1)
	\end{align*}

	Now, for fixed $\epsilon >0$ we take $h_{\epsilon} = \text{e}^{-x}
	\epsilon^{3/2} \eta (\epsilon |y|) $ with $\eta \in
	H^{5/2}(\mathbb{R}^{3}; \C^2 )$, $\eta(y) = \eta(\abs{y})$ and
	$|\eta|_{L^2} = 1$.

	Note that 
	\begin{equation*}
		\| h_{\epsilon} \|^2_{H^1} = \frac{1}{2} \epsilon^{2} |\nabla
		\eta |^2_{L^2} + |\eta |^2_{L^2}
	\end{equation*}
	and setting $\xi_{\epsilon} = U_{_{\text{FW}}}^{-1} \left(
	\begin{smallmatrix} (h_{\epsilon})_{tr} \\ 0\end{smallmatrix}
	\right)$ we have
	\begin{equation*}
		\xi_{\epsilon}(y) = \epsilon^{3/2} U_{_{\text{FW}}}^{-1}
		\left( \begin{smallmatrix} \eta (\epsilon |y|) \\
		0\end{smallmatrix} \right) = \epsilon^{3/2} \left(
		\begin{smallmatrix} \mathcal{F}^{-1} [ a_{+}(\epsilon p)
		\hat{\eta}] \\ \mathcal{F}^{-1} [ a_{-} (\epsilon p)
		\frac{\sigma \cdot p}{|p|} \hat{\eta}] \end{smallmatrix}
		\right) (\epsilon y)
	\end{equation*}
	hence in particular
	\begin{align*}
		\rho_{ \xi_{\epsilon}}(y) = & \epsilon^3 | \mathcal{F}^{-1} [
		a_{+}(\epsilon p) \hat{\eta}]|^2 (\epsilon y) + \epsilon^3 |
		\mathcal{F}^{-1}[ a_{-} (\epsilon p) \frac{\sigma \cdot
		p}{|p|} \hat{\eta}]|^2(\epsilon y) \\
		=& \epsilon^3 \rho_{\eta}(\epsilon |y|) + \epsilon^3
		\zeta_{\epsilon} (\epsilon y)
	\end{align*}
	where $\rho_{\eta} (y)= |\eta|^2(y)$ and
	\begin{align*}
		\zeta_{\epsilon} ( y) =& | \mathcal{F}^{-1} [ (a_{+}(\epsilon
		p) -1) \hat{\eta}]|^2 (y) + | \mathcal{F}^{-1}[ a_{-}
		(\epsilon p) \frac{\sigma \cdot p}{|p|} \hat{\eta}]|^2(y) \\
		& + 2 \RE(\eta, \mathcal{F}^{-1} [ (a_{+}(\epsilon p) -1)
		\hat{\eta} ]) (y)
	\end{align*}
	 Recalling that $a_{\pm}(\epsilon p) = \sqrt{\frac{1}{2}( 1 \pm 1
	 / \lambda(\epsilon p))} $, and $\lambda(p)= \sqrt{ |p|^2 + 1} $,
	 we have
	\begin{equation*}
		|a_{+} (\epsilon p) - 1| = | \left(\frac{\lambda(\epsilon p) +
		1 }{2 \lambda(\epsilon p)} \right)^{\frac{1}{2}} - 1 | \leq
		\frac{|1 - \lambda(\epsilon p)|}{2 \lambda(\epsilon p)} \leq
		\epsilon^2 |p|^2
	\end{equation*}
	\begin{equation*}
		| a_{-} (\epsilon p) | = \left(\frac{\lambda(\epsilon p) - 1
		}{2 \lambda(\epsilon p)} \right)^{\frac{1}{2}} \leq \epsilon
		|p|
	\end{equation*}
	we have
	\begin{equation}
		\label{eq:stimeApm}
		\begin{split}
			&| (a_{+}(\epsilon p) -1) \hat{\eta} |_{L^2} \leq C
			\epsilon^2 | |p|^2 \hat{\eta} |_{L^2} \\
			&| a_{-} (\epsilon p) \hat{\eta} |_{L^2} \leq C \epsilon |
			|p| \hat{\eta} |_{L^2}
		\end{split}
	\end{equation}
	Therefore we get
	\begin{align*}
		d^2 \mathcal{F}&(w_n) [h_{\epsilon} ; h_{\epsilon}] \leq 2
		a_n^2 ( 1 - \mult_n) \\
		&+ 2 a_n^2 \cfstr \int_{\R^{3}} V(y) \rho_{\eta}(\epsilon |y|) \,
		\epsilon^3 dy + 8 a_n^2 \cfstr \iint_{\R^3 \times \R^3 }
		\frac{\rho_{\psi_n}(y) \rho_{\eta}(\epsilon |z|) \, \epsilon^3
		} {|y-z | } \, dy \, dz \\
		& + 2 a_n^2 \cfstr \int_{\R^{3}} V(y)
		\zeta_{\epsilon}(\epsilon y) \, \epsilon^3 dy + 8 a_n^2 \cfstr
		\iint_{\R^3 \times \R^3 } \frac{\rho_{\psi_n}(y)
		\zeta_{\epsilon}(\epsilon z) \, \epsilon^3 } {|y-z | } \, dy
		\, dz \\
		& + C \epsilon^2 | \nabla \eta |^2_{L^2} + o_n(1)
	\end{align*}
	Here, using \eqref{eq:Kato} we get
	\begin{align*}
		\int_{\R^{3}} V(y) \zeta_{\epsilon}(\epsilon y)\, \epsilon^3
		dy &= -Z \epsilon \int_{\R^{3}} \frac{\zeta_{\epsilon}(
		y)}{\abs{y}} \, dy \\
		&= -Z \epsilon \int_{\R^{3}} \frac{ \abs{ \mathcal{F}^{-1} [
		(a_{+}(\epsilon p) -1) \hat{\eta}]}^2 (y)}{\abs{y}} \, dy \\
		&\quad -Z \epsilon \int_{\R^{3}} \frac{ \abs {
		\mathcal{F}^{-1}[ a_{-} (\epsilon p) \frac{\sigma \cdot
		p}{|p|} \hat{\eta}]}^2(y)}{\abs{y}} \, dy \\
		&\quad - 2 Z \epsilon \int_{\R^{3}} \frac{\RE(\eta,
		\mathcal{F}^{-1} [ (a_{+}(\epsilon p) -1) \hat{\eta} ])
		(y)}{\abs{y}} \, dy \\
		&\leq \epsilon C \abs{\mathcal{F}^{-1}[(a_{+}(\epsilon p) -1)
		\hat{\eta}}_{H^{1/2}}^{2}
		+ \epsilon C \abs{\mathcal{F}^{-1}[
		a_{-} (\epsilon p) \frac{\sigma \cdot p}{|p|}
		\hat{\eta}]}_{H^{1/2}}^{2} \\
		&\qquad  + \epsilon C \abs{\mathcal{F}^{-1}[(a_{+}(\epsilon p) -1)
		\hat{\eta}]}_{H^{1/2}} \abs{\eta}_{H^{1/2}}
		\\
		&\leq \epsilon^{4} C \abs{\abs{p}^{5/2}
		\hat{\eta}]}_{L^{2}}^{2} + \epsilon^{3} C \abs{\abs{p}^{3/2}
		\hat{\eta}]}_{L^{2}}^{2} + \epsilon^{3} C \abs{\abs{p}^{5/2}
		\hat{\eta}]}_{L^{2}} \abs{\eta}_{H^{1/2}}
	\end{align*}
	Since for any radial function $\rho \in L^1(\R^3; \R_{+})$ and for
	any $z \in \R^3$ we have
	\begin{equation*}
		\int_{\R^3 } \frac{ \rho(y) } {|y - z|} \, dy \leq \int_{\R^3}
		\frac{\rho(y)} {|y| } \, dy.
	\end{equation*}
	we deduce
	\begin{multline*}
		\iint_{\R^3 \times \R^3 } \frac{\rho_{\psi_n}(y)
		\rho_{\eta}(\epsilon |z|) \, \epsilon^3 } {|y-z | } \, dy \,
		dz = \int_{\mathbb{R}^{3}} \rho_{\psi_{n}}(y) \left(
		\int_{\mathbb{R}^{3}} \frac{\rho_{\eta}(\epsilon\abs{z})
		\epsilon^{3}}{\abs{y - z}} \, dz \right) \, dy \\
		\leq \epsilon \abs{\rho_{\psi_{n}}}_{L^{1}}
		\int_{\mathbb{R}^{3}} \frac{\rho_{\eta}(\abs{z})}{\abs{z}} \,
		dz
	\end{multline*}
	and, by Lemma \ref{lem:key_estimate}
	\begin{equation*}
		\iint_{\R^3 \times \R^3 } \frac{\rho_{\psi_n}(y)
		\zeta_{\epsilon}(\epsilon z) \, \epsilon^3 } {|y-z | } \, dy
		\, dz \leq C \abs{\psi_{n}}_{H^{1/2}}^{2}
		\abs{\zeta_{\epsilon}}_{L^{1}} \leq C
		\epsilon^{2}\abs{\eta}_{H^{2}}^{2} + o(\epsilon^{2})
	\end{equation*}
	Then, by \eqref{eq:Hessian_condition} and Lemma
	\ref{lem:key_estimate} we get
	\begin{align*}
		0 \leq & \liminf_{n \to + \infty} d^2 \mathcal{F}(w_n)
		[h_{\epsilon} ; h_{\epsilon}] \\
		&\leq 2 a^2 ( 1 - \mult) + 2 a^2 \cfstr \int_{\R^{3}}
		\left(-\frac{Z\epsilon}{\abs{y}} + \frac{ 4 \epsilon} {|y | }
		\right) \rho_{\eta}( |y|) \,dy \\
		&\quad + C \epsilon^2 | \eta |^2_{H^2} + o(\epsilon^{2})
	\end{align*}
	where $a = \lim_{n} a_{n}$ (up to subsequence).  	
	
	Hence, since $Z > 4$ we may conclude that there exists
	$\bar{\delta} >0$ and $\bar{\epsilon} > 0$ such that
	\begin{align*}
		0 \leq ( 1 - \mult) - \epsilon \bar{\delta} \int_{\R^{3}}
		\frac{ |\eta|^2 } {|y | } \, \,dy + C \epsilon^2 | \eta
		|^2_{H^2} \leq 1 - \mult - C \bar{\epsilon}
	\end{align*}
	where we have denoted with $C$ various positive constants.
\end{proof}

Now, let $h_{n} =\left( \begin{smallmatrix} h_{1,n} \\ h_{2,n}
\end{smallmatrix} \right) = \phi_n - \phi \rightharpoonup 0 $ weakly
in $H^1$, define $\xi_{+,n} = \fw^{-1} \left( \begin{smallmatrix}
(h_{1,n})_{tr} \\ 0 \end{smallmatrix} \right) $ and $\xi_{-,n} =
\fw^{-1} \left( \begin{smallmatrix} 0 \\ (h_{2,n})_{tr}
\end{smallmatrix} \right) $ and $\xi_n = \fw^{-1} (h_n)_{tr} =
\xi_{+,n} + \xi_{-,n}$ we have
\begin{lem}
	\label{lem:vanishing2}
	If $\xi_n \rightharpoonup 0$ weakly in $H^{1/2}$ then
	\begin{equation*}
		\int_{\R^3 } V \rho_{\xi_{\pm, n} } \, dy \to 0 .
	\end{equation*}
\end{lem}

\begin{proof}
	The proof is similar, even somewhat simpler than the \cite[Lemma
	B.1]{CotiZelati_Nolasco_2016}
\end{proof}

Then finally taking $\zeta_{n} = \beta h_n $.  in view of Corollary
\ref{cor:sup_achieved_I}, and since $h_n = \phi_n - \phi
\rightharpoonup 0$ by Lemma \ref{lem:vanishing2} , we get
\begin{align*}
	o_n(1) &= \mathcal{T}_n(\zeta_n) = d \mathcal{I}(\phi_n) [\zeta_n]
	- 2 \mult_{n} \RE \langle (\psi_n, \xi_{+,n}-\xi_{-,n}
	\rangle_{L^2} \\
	& = 2\| h_{1,n} \|^2_{H^1} + 2\| h_{2,n} \|^2_{H^1} - 2 \mult_{n}
	( | \xi_{+, n} |^2_{L^2} - |\xi_{-, n}|^2_{L^2} ) \\
	&\quad + 2 \cfstr \int_{\R^3 } V ( \rho_{\xi_{+, n}} - \rho_{
	\xi_{-, n}} ) \, dy \\
	&\quad + 2 \cfstr \iint_{\R^3 \times \R^3 } \frac{ \rho_{\psi_n}
	(y) \rho_{\xi_{+, n}} (z) - J_{\psi_n} (y) \cdot J_{ \xi_{+, n}}
	(z) } {|y-z | } \, dy \, dz\\
	&\quad - 2 \cfstr \iint_{\R^3 \times \R^3 } \frac{ \rho_{\psi_n}
	(y) \rho_{\xi_{-, n}}(z) - J_{\psi_n} (y) \cdot J_{ \xi_{-, n}
	}(z)} {|y-z | } \, dy \, dz + o_n(1) \\
	&\geq 2 (1 - \mult_n ) \| h_{1,n} \|^2_{H^1} + 2(1 - 2 \gamma_T)
	\| h_{2,n} \|^2_{H^1} + o_n(1) .
\end{align*}
since $\mult < 1$ we may conclude that $\phi_n \to \phi $ strongly
in $H^1$.

\providecommand{\noopsort}[1]{} \providecommand{\cprime}{$'$}
  \providecommand{\scr}{} \def\ocirc#1{\ifmmode\setbox0=\hbox{$#1$}\dimen0=\ht0
  \advance\dimen0 by1pt\rlap{\hbox to\wd0{\hss\raise\dimen0
  \hbox{\hskip.2em$\scriptscriptstyle\circ$}\hss}}#1\else {\accent"17 #1}\fi}


\begin{thebibliography}{10}

\bibitem{zbMATH06266957}
L.~{Ambrosio} and J.~{Feng}.
\newblock {On a class of first order Hamilton-Jacobi equations in metric
  spaces.}
\newblock {\em {J. Differ. Equations}}, 256(7):2194--2245, 2014.

\bibitem{zbMATH04028246}
J.~M. {Borwein} and D.~{Preiss}.
\newblock A smooth variational principle with applications to
  subdifferentiability and to differentiability of convex functions.
\newblock {\em {Trans. Am. Math. Soc.}}, 303:517--527, 1987.

\bibitem{CabreMorales05}
X.~Cabr{\'e} and J.~Sol{\`a}-Morales.
\newblock Layer solutions in a half-space for boundary reactions.
\newblock {\em Comm. Pure Appl. Math.}, 58(12):1678--1732, 2005.

\bibitem{CZNolasco2011}
V.~Coti~Zelati and M.~Nolasco.
\newblock Existence of ground states for nonlinear, pseudorelativistic
  {S}chr\"odinger equations.
\newblock {\em Rend. Lincei Mat. Appl.}, 22:51--72, 2011.

\bibitem{CZNolasco2013}
V.~Coti~Zelati and M.~Nolasco.
\newblock Ground states for pseudo-relativistic {H}artree equations of critical
  type.
\newblock {\em Rev. Mat. Iberoam.}, 29(4):1421--1436, 2013.

\bibitem{CotiZelati_Nolasco_2016}
V.~{Coti Zelati} and M.~{Nolasco}.
\newblock A variational approach to the {B}rown-{R}avenhall operator for the
  relativistic one-electron atoms.
\newblock {\em {Nonlinear Anal., Theory Methods Appl., Ser. A, Theory
  Methods}}, 136:62--83, 2016.

\bibitem{Dolbeault_Esteban_Sere_2000}
J.~Dolbeault, M.~J. Esteban, and E.~S{\'e}r{\'e}.
\newblock Variational characterization for eigenvalues of {D}irac operators.
\newblock {\em Calc. Var. Partial Differential Equations}, 10(4):321--347,
  2000.

\bibitem{Georgiev_Esteban_Sere_1996}
M.~J. Esteban, V.~Georgiev, and E.~S\'er\'e.
\newblock Stationary solutions of the {M}axwell-{D}irac and the
  {K}lein-{G}ordon-{D}irac equations.
\newblock {\em Calc. Var. Partial Differential Equations}, 4(3):265--281, 1996.

\bibitem{EstebanSere99}
M.~J. Esteban and E.~S{\'e}r{\'e}.
\newblock Solutions of the {D}irac-{F}ock equations for atoms and molecules.
\newblock {\em Comm. Math. Phys.}, 203(3):499--530, 1999.

\bibitem{LiebLoss}
E.~H. Lieb and M.~Loss.
\newblock {\em Analysis}.
\newblock Number~14 in Graduate Studies in Mathematics. American Mathematical
  Society, 1997.

\bibitem{Morozov_Muller2015}
S.~Morozov and D.~M\"{u}ller.
\newblock On the minimax principle for {C}oulomb-{D}irac operators.
\newblock {\em Math. Z.}, 280(3-4):733--747, 2015.

\bibitem{Morozov_Vugalter_2006}
S.~Morozov and S.~Vugalter.
\newblock Stability of atoms in the {B}rown-{R}avenhall model.
\newblock {\em Ann. Henri Poincar\'e}, 7(4):661--687, 2006.

\bibitem{Thaller_1992}
B.~Thaller.
\newblock {\em The {D}irac equation}.
\newblock Texts and Monographs in Physics. Springer-Verlag, Berlin, 1992.

\bibitem{Tix_1998}
C.~Tix.
\newblock Strict positivity of a relativistic {H}amiltonian due to {B}rown and
  {R}avenhall.
\newblock {\em Bull. London Math. Soc.}, 30(3):283--290, 1998.

\end{thebibliography}
\end{document}